\documentclass{amsart}

\usepackage{amsmath,amsfonts,amssymb,amsthm,amscd,comment,euscript}
\usepackage{stmaryrd}
\usepackage[all,cmtip]{xy}
\usepackage{graphicx}
\usepackage{enumerate} 
\usepackage[colorlinks=true]{hyperref} 
\usepackage[usenames,dvipsnames]{xcolor}
\usepackage{accents}

\allowdisplaybreaks[4]

\newcommand{\equalby}[1]{\overset{\textrm{#1}}=}             
\newcommand{\equalbyeq}[1]{\equalby{\eqref{#1}}}             
\newcommand{\equalbyref}[1]{\equalby{\ref{#1}}}                
\newcommand{\isoto}{\buildrel \sim\over\to}                 

\newcommand{\Z}{\mathbb{Z}}               
\newcommand{\Q}{\mathbb{Q}}              

\newcommand{\Kr}{\delta^{\text{Kr}}}               

\newcommand{\ot}{\otimes}                 
\newcommand{\id}{\mathrm{id}}           

\newcommand{\ie}{i.e.\ }

\DeclareMathOperator{\Hom}{\mathrm{Hom}}
\DeclareMathOperator{\End}{\mathrm{End}}

\newcommand{\al}{\alpha}      
\newcommand{\la}{\lambda}   
\newcommand{\RS}{\Sigma}    
\newcommand{\cl}{\Lambda}  

\newcommand{\fplus}{+_F} 

\newcommand{\lbr}{[\hspace{-1.5pt}[}  
\newcommand{\rbr}{]\hspace{-1.5pt}]}  
\newcommand{\FGR}[3]{#1 \lbr #2 \rbr_{#3}} 
\newcommand{\RcF}{\FGR{R}{\cl}{F}} 
\newcommand{\aug}{\epsilon}        
\newcommand{\IF}{\mathcal{I}_F}    

\newcommand{\SW}{S_W} 
\newcommand{\SWd}{S_W^\star}  
\newcommand{\QW}{Q_W}  
\newcommand{\QWd}{Q_W^*} 
\newcommand{\SWP}[1]{S_{W/W_{#1}}}  
\newcommand{\QWP}[1]{Q_{W/W_{#1}}}  

\newcommand{\SWPd}[1]{S_{W/W_{#1}}^\star}  
\newcommand{\QWPd}[1]{Q_{W/W_{#1}}^*}  %

\newcommand{\p}[1]{p_{#1}}   
\newcommand{\dd}[1]{d_{#1}}   

\newcommand{\pd}[1]{p_{#1}^\star}   
\newcommand{\ddd}[1]{d_{#1}^\star}   %

\newcommand{\pqd}[1]{p_{#1}^*}   
\newcommand{\ddqd}[1]{d_{#1}^*}   %

\newcommand{\cA}{\mathcal{A}}        
\newcommand{\Xa}[2]{\hat {X}^{#1}_{#2}} 
\newcommand{\Ya}[2]{\hat {Y}^{#1}_{#2}} 
\newcommand{\de}{\delta}                   
\newcommand{\Dem}{\Delta}                  
\newcommand{\ED}{\mathcal{D}(\cl)_F} 
\newcommand{\DcF}{\mathbf{D}_F}    
\newcommand{\DcFd}{\mathbf{D}_F^\star}   
\newcommand{\DcFP}[1]{\mathbf{D}_{F,#1}} 
\newcommand{\DcFPd}[1]{\mathbf{D}_{F,#1}^\star} 
\newcommand{\mZ}{\mathcal{Z}}     

\newcommand{\kp}{\kappa}             

\newcommand{\QWcopr}{\triangle} 
\newcommand{\QWcoun}{\varepsilon} 

\newcommand{\unit}{\mathbf{1}}      

\newcommand{\act}{\bullet}            
\newcommand{\rev}{\mathrm{rev}}   
\newcommand{\conc}{\cup}             

\newcommand{\tf}{\tilde{f}} 

\DeclareMathOperator{\hh}{\mathtt{h}}             
\newcommand{\pt}{\mathrm{pt}}   		

\theoremstyle{plain}
\newtheorem{theo}{Theorem}[section]
\newtheorem{prop}[theo]{Proposition}
\newtheorem{lem}[theo]{Lemma}
\newtheorem{cor}[theo]{Corollary}

\theoremstyle{definition}
\newtheorem{defi}[theo]{Definition}
\newtheorem{rem}[theo]{Remark}
\newtheorem{example}[theo]{Example}

\numberwithin{equation}{section}   

\begin{document}
%

\title[Push-pull operators]{Push-pull operators on the formal affine Demazure algebra and its dual}

\author{Baptiste Calm\`es}
\author{Kirill Zainoulline}
\author{Changlong Zhong}

\address{Baptiste Calm\`es, Universit\'e d'Artois, Laboratoire de
  Math\'ematiques de Lens, France}
\email{baptiste.calmes@univ-artois.fr}

\address{Kirill Zainoulline, Department of Mathematics and Statistics,
University of Ottawa, Canada}
\email{kirill@uottawa.ca}

\address{Changlong Zhong, Department of Mathematical and Statistical Sciences,
University of Alberta, Canada}
\email{zhongusc@gmail.com}

\thanks{The first author acknowledges the support of the French Agence Nationale de la Recherche (ANR) under reference ANR-12-BL01-0005.
The second author was supported by the NSERC Discovery grant  385795-2010, NSERC DAS grant 396100-2010 and the Early Researcher Award (Ontario). We appreciate the support of the Fields Institute; part of this work was done while authors were attending the Thematic Program on Torsors, Nonassociative algebras and Cohomological Invariants at the Fields Institute. 
} 


\maketitle


\setcounter{tocdepth}{3}
\tableofcontents

\section{Introduction}

In a series of papers \cite{KK86}, \cite{KK90} Kostant and Kumar introduced and successfully applied the techniques of nil (or $0$-) Hecke algebras to study equivariant cohomology and K-theory of flag varieties. In particular, they showed that the dual of the nil Hecke algebra serves as an algebraic model for the $T$-equivariant singular cohomology of $G/B$ (here $G$ is a split semisimple linear algebraic group with a chosen split maximal torus $T$ and $G/B$ is the variety of Borel subgroups). In \cite{HMSZ} and \cite{CZZ}, this formalism has been generalized using an arbitrary formal group law associated to an algebraic oriented cohomology theory in the sense of Levine-Morel \cite{levmor-book}, via the Quillen formula. Namely, given a formal group law $F$ and a finite root system with a set of simple roots $\Pi$, one defines the \emph{formal affine Demazure algebra} $\DcF$ and its dual $\DcFd$ provides an algebraic model for the $T$-equivariant oriented cohomology $\hh_T(G/B)$. Specializing to the additive and the multiplicative formal group laws, one recovers Chow groups (or singular cohomology) and $K$-theory respectively.

Another motivation for studying the algebra $\DcF$ comes from its close relationship to Hecke algebras. Indeed, for the additive (resp. multiplicative) $F$ it coincides with the completion of the nil (resp. 0-) affine Hecke algebra (see \cite{HMSZ}). Moreover, in section~\ref{sec:iwahorihecke}, we show that for some elliptic formal group law $F$ and a root system of Dynkin type $A$ the non-affine part of $\DcF$ is isomorphic to the classical Iwahori-Hecke algebra, hence, relating it to equivariant elliptic cohomology.

In the present paper we pursue the `algebraization program' for oriented cohomology theories started in \cite{CPZ} and continued in \cite{HMSZ} and \cite{CZZ}; the general idea is to match cohomology rings of flag varieties and elements of classical interest in them (such as classes of Schubert varieties) with  algebraic and combinatorial objects that can be introduced simply and algebraically, in the spirit of \cite{Dem73} or \cite{KK86}. This approach is useful to study the structure of these rings, and to perform various computations. We focus here on algebraic constructions pertaining to $T$-equivariant oriented cohomology groups. The precise proofs and details of how our algebraic objects match cohomology groups will be given in \cite{CZZ2}; however, for the convenience of the reader, we now give a brief description of the geometric setting. 

Given an equivariant oriented cohomology theory $\hh$ over a base field whose spectrum is denoted by $\pt$, the formal group algebra $S$ will correspond to $\hh_T(\pt)$.%
\footnote{We will require that the cohomology rings are `complete' in some precise sense, but this is a technical point, that we prefer to hide here for simplicity. See \cite[Definition 2.1]{CZZ2}}
It is an algebra over $R=\hh(\pt)$. 

The $T$-fixed points of $G/B$ are naturally in bijection with the Weyl group $W$. This gives a pull-back to the fixed locus map $\hh_T(G/B) \to \hh_T(W)\simeq \bigoplus_{w \in W} \hh_T(\pt)$. This map happens to be injective. We do not know a direct geometric reason for that, but it follows from our algebraic description, in which it appears as the map $\DcFd \to S^\star_W \simeq \bigoplus_{w \in W} S$ of Definition \ref{defi:algres}. It is then convenient to enlarge $S$ to its localization $Q$ at a multiplicative subset generated by Chern classes of line bundles corresponding canonically to roots, which gives injections $S \subseteq Q$, $\SW \subseteq \QW$ and $\SWd \subseteq \QWd$. Although we do not know good geometric interpretations of $Q$, $\QW$ or $\QWd$, all the formulas and operators we are interested in are easily defined at that localized level, because they involve denominators. The main technical difficulties then lie in proving that these operators actually restrict to $S$, $\SWd$, $\DcFd$ etc., or so to speak, that the denominators cancel out. 

Our central object of study is a push-pull operator on $\DcFd$, which is an algebraic version of the composition
\[
\hh_T(G/P) \stackrel{p_*}\to \hh_T(G/Q) \stackrel{p^*}\to \hh_T(G/P)
\]
of the push-forward followed by the pull-back along the quotient map  $p\colon G/P\to G/Q$, where $P\subseteq Q$ are two parabolic subgroups of $G$. Again $p^*$ happens to be injective, and it identifies $\hh_T(G/Q)$ to a subring of $\hh_T(G/P)$, namely the subring of invariants under the action of the parabolic subgroup $W_Q$ of the Weyl group $W$. This does not seem to be straightforward from the geometry either, and it once more follows from our algebraic description: given subsets $\Xi' \subseteq \Xi$ of a given set of simple roots $\Pi$ (each giving rise to a parabolic subgroup), we define an element $Y_{\Xi/\Xi'}$ in $\QW$ (see \ref{defi:CY}). We define an action of the Demazure algebra $\DcF$ on its $S$-dual $\DcFd$, by precomposition by multiplication on the right. The action of $Y_{\Xi/\Xi'}$ thus defines the desired push-pull operator $A_{\Xi/\Xi'}:(\DcFd)^{W_{\Xi'}} \to (\DcFd)^{W_\Xi}$. The formula for the element $Y_{\Xi/\Xi'}$ with $\Xi'=\emptyset$ had already appeared in related contexts, namely, in discussions around the Becker-Gottlieb transfer for topological complex-oriented theories (see \cite[(2.1)]{BE90} and \cite[\S4.1]{GaRa}).

Finally, we define the algebraic counterpart of the natural pairing $\hh_T(G/B) \ot \hh_T(G/B) \to \hh_T(\pt)$ obtained by multiplication and push-forward to the point. It is a pairing $\DcFd \ot \DcFd \to S$. We show that it is non-degenerate, and that algebraic classes corresponding to (chosen) desingularization of Schubert varieties form a basis of $\DcFd$, with a very simple dual basis with respect to the pairing. We provide the same kind of description for $\hh_T(G/P)$. This generalizes (to parabolic subgroups and to equivariant cohomology groups) and simplifies several statements from \cite[\S14]{CPZ}, as well as results from \cite{KK86} and \cite{KK90} (to arbitrary oriented cohomology theories).
\medskip

The paper is organized as follows.
In sections~\ref{sec:fode} and \ref{sec:formalaffineDemazure}, we recall definitions and basic properties from \cite[\S2, \S3]{CPZ}, \cite[\S6]{HMSZ} and \cite[\S4, \S5]{CZZ}: the formal group algebra $S$, the Demazure and push-pull operators $\Dem_\al$ and $C_\al$ for every root $\al$, the formal twisted group algebra $\QW$ and its Demazure and push-pull elements $X_\al$ and $Y_\al$.
In section~\ref{sec:weylhecke}, we introduce a left $\QW$-action `$\act$' on the dual $\QWd$. It induces both an action of the Weyl group $W$ on $\QWd$ (the Weyl-action) and an action of $X_\al$ and $Y_\al$ on $\QWd$ (the Hecke-action).
In sections~\ref{sec:pushpull} and \ref{sec:pushpulldual}, we introduce and study more general push-pull elements in $\QW$ and operators on $\QWd$ with respect to given coset representatives of parabolic quotients of the Weyl group.
In section~\ref{sec:coeff} we study relationships between some technical coefficients.
In section~\ref{sec:anotherbasis}, we construct a basis of the subring of invariants of $\QWd$, which generalizes \cite[Lemma 2.27]{KK90}.

In section~\ref{sec:iwahorihecke}, we recall the definition and basic properties of the formal (affine) Demazure algebra $\DcF$ following \cite[\S6]{HMSZ}, \cite[\S5]{CZZ} and \cite{Zh}. We show that for a certain elliptic formal group law (Example~\ref{elliptic_ex}), the formal Demazure algebra can be identified with the classical Iwahori-Hecke algebra.
In section~\ref{sec:algmom}, we define the algebraic restriction to the fixed locus map which is used in section~\ref{sec:pushpullDdual} to restrict all our push-pull operators and elements to $\DcF$ and its dual $\DcFd$ as well as to restrict the non-degenerate pairing on $\DcFd$.
In section~\ref{sec:algresP}, we define the algebraic restriction to the fixed locus map on $G/P$ for any parabolic subgroup $P$.
In section~\ref{sec:invol}, we define an involution on $\DcFd$ which is used to relate the equivariant characteristic map with the push-pull operators. In section~\ref{sec:nondeg}, we define and discuss the non-degenerate pairing on the subring of invariants of $\DcFd$ under a parabolic subgroup of the Weyl group. 
At last, in section~\ref{sec:pushDXi}, in the parabolic case, we identify the Weyl group invariant subring $(\DcFd)^{W_\Xi}$ with $\DcFPd{\Xi}$, the dual of a quotient of $\DcF$, which matches more naturally to $\hh_T(G/P)$.

\medskip

\noindent
{\it Acknowledgments:}
One of the ingredients of this paper, the push-pull formulas in the context of Weyl group actions, arose in discussions between the first author and Victor Petrov, whose unapparent contribution we therefore gratefully acknowledge.

\section{Formal Demazure and push-pull operators}\label{sec:fode}

In this section we recall definitions of the formal group algebra and of the formal Demazure and push-pull operators, following \cite[\S2, \S3]{CPZ} and \cite{CZZ}. 

\medskip

Let $R$ be a commutative ring with unit, and let $F$ be a one-dimensional commutative \emph{formal group law} (FGL) over $R$, \ie $F(x,y)\in R\lbr x, y\rbr$ satisfies
\[
F(x,0)=0,\; F(x,y)=F(y,x) \text{ and } F(x,F(y,z))=F(F(x,y), z).
\]
\begin{example} 
The \emph{additive} FGL is defined by $F_a(x,y)=x+y$, and a \emph{multiplicative} FGL is defined by $F_m(x,y)=x+y-\beta xy$ with $\beta\in R$. The coefficient ring of the {\em universal} FGL $F_u(x,y)=x+y+\sum_{i,j\ge 1}a_{i,j}x^iy^j$ is generated by the coefficients $a_{ij}$ modulo relations induced by the above properties and is called the {\em Lazard ring}.
\end{example}

\begin{example} \label{elliptic_ex} 
Consider an elliptic curve given in Tate coordinates by 
\[
(1-\mu_1 t-\mu_2 t^2)s=t^3.
\] 
The corresponding FGL over the coefficient ring $R=\Z[\mu_1,\mu_2]$ is given by \cite[Cor.~2.8]{BB10}
\[
F(x,y):=\tfrac{x+y-\mu_1xy}{1+\mu_2xy}.
\] 
Its genus is the 2-parameter generalized Todd genus introduced and studied by Hirzebruch in \cite{Hi66}.
Its exponent is given by the rational function $\tfrac{e^{\epsilon_1 x}+e^{\epsilon_2 x}}{\epsilon_1e^{\epsilon_1 x}+\epsilon_2 e^{\epsilon_2 x}}$
in $e^x$, where $\mu_1=\epsilon_1+\epsilon_2$ and $\mu_2=-\epsilon_1\epsilon_2$ which suggests to call $F$ a \emph{hyperbolic} FGL and to denote it by $F_h$.

By definition we have 
\[
F_h(x,y)=x+y-xy(\mu_1+\mu_2F_h(x,y))
\]
and, thus, that the formal inverse of $F_h$ is identical to the one of $F_m$ (\ie $\tfrac{x}{\mu_1x-1}$) and $F_h(x,x)=\tfrac{2x-\mu_1x^2}{1+\mu_2x^2}$.
\end{example} 

Let $\cl$ be an Abelian group and let $R\lbr x_\cl\rbr$ be the ring of formal power series with variables $x_\la$ for all $\la \in \cl$. Define the \emph{formal group algebra} $S:=\RcF$ to be the quotient of $R\lbr x_\cl\rbr$ by the closure of the ideal generated by elements $x_0$ and $x_{\la_1+\la_2}-F(x_{\la_1}, x_{\la_2})$ for any $\la_1,\la_2\in \cl$. Here $0$ is the identity element in $\cl$. Let $\IF$ denote the kernel of the augmentation map $\aug\colon S\to R$, $x_\al\mapsto 0$.

\medskip

Let $\cl$ be a free Abelian group of finite rank and let $\RS$ be a finite subset of $\cl$. A \emph{root datum} is an embedding $\RS\hookrightarrow \cl^\vee$, $\al\mapsto \al^\vee$ into the dual of $\cl$ satisfying certain conditions \cite[Exp. XXI, Def. 1.1.1]{SGA}. The \emph{rank} of the root datum is the $\Q$-rank of $\cl\ot_\Z\Q$. The \emph{root lattice} $\cl_r$ is the subgroup of $\cl$ generated by $\RS$, and the \emph{weight lattice} $\cl_w$ is the Abelian group defined by 
\[
\cl_w:=\{\omega\in \cl\ot_\Z\Q\mid \al^\vee(\omega)\in \Z \text{ for all } \al\in \RS\}.
\]
We always assume that the root datum is reduced and \emph{semisimple} ($\Q$-ranks of $\cl_r$, $\cl_w$ and $\cl$ are the same and no root is twice another one). We say that a root datum is \emph{simply connected} (resp. \emph{adjoint}) if $\cl=\cl_w$ (resp. $\cl=\cl_r$), and then use the notation $\mathcal{D}^{sc}_n$ (resp. $\mathcal{D}^{ad}_n$) for irreducible root data where $\mathcal{D}=A, B, C, D, E, F, G$ is one of the Dynkin types and $n$ is the rank.

\medskip

The \emph{Weyl group} $W$ of a root datum $(\cl, \RS)$ is a subgroup of $\mathrm{Aut}_\Z(\cl)$ generated by simple reflections $s_\al$ for all $\al\in \RS$ defined by 
\[
s_\al(\la):=\la-\al^\vee(\la)\al,\quad \la\in \cl.
\] 
We fix a set of \emph{simple roots} $\Pi=\{\al_1,\ldots,\al_n\}\subset \RS$, \ie a basis of the root datum: each element of $\RS$ is an integral linear combination of simple roots with either all positive or all negative coefficients. This partitions $\RS$ into the subsets $\RS^+$ and $\RS^-$ of \emph{positive} and \emph{negative} roots. Let $\ell$ denote the \emph{length function} on $W$ with respect to the set of simple roots $\Pi$. Let $w_0$ be the \emph{longest element} of $W$ with respect to $\ell$ and let $N:=\ell(w_0)$. 
 
\medskip

Following \cite[Def.~4.4]{CZZ} we say that the formal group algebra $S$ is \emph{$\RS$-regular} if $x_\al$ is not a zero divisor in $S$ for all roots $\al\in \RS$. We will always assume that:
\begin{quote}
{\em The formal group algebra $S$ is $\RS$-regular.}
\end{quote}
By \cite[Lemma~2.2]{CZZ} this holds if $x \fplus x$ is not a zero divisor in $R\lbr x \rbr$, in particular if $2$ is not a zero divisor in $R$, or if the root datum does not contain any symplectic datum $C^{sc}$ as an irreducible component. 

\medskip

Following \cite[Definitions~3.5 and 3.12]{CPZ} for each $\al\in \RS$ we define two $R$-linear operators $\Dem_\al$ and $C_\al$ on $S$ as follows: 
\begin{equation}\label{eq:opdec}
\Dem_{\al}(y):=\tfrac{y-s_\al(y)}{x_\al}, \quad C_\al(y):=\kp_\al y-\Dem_\al(y)=\tfrac{y}{x_{-\al}}+\tfrac{s_\al(y)}{x_\al}, \quad  y\in S,
\end{equation}
where $\kp_\al:=\frac{1}{x_\al}+\frac{1}{x_{-\al}}$ (note that $\kp_\al\in S$). The operator $\Dem_\al$ is called the \emph{Demazure} operator and the operator $C_\al$ is called the \emph{push-pull} operator or the \emph{BGG} operator.

\begin{example}
For the hyperbolic formal group law $F_h$ we have $\kp_\al=\mu_1+\mu_2 F_h(x_{-\al},x_\al)=\mu_1$ for each $\al\in \RS$.  If the root datum is of type $A_1^{sc}$, we have $\RS=\{\pm \al\}$, $\cl=\langle \omega \rangle$ with simple root $\al=2\omega$ and 
\[
C_\al(x_\al)=\tfrac{x_\al}{x_{-\al}}+\tfrac{x_{-\al}}{x_\al}=\mu_1x_\al -1 +\tfrac{1}{\mu_1x_\al-1},\quad C_\al(x_\omega) = \tfrac{x_\omega}{x_{-\al}}+\tfrac{x_{-\omega}}{x_\al} = \mu_1 x_\omega - \tfrac{1+\mu_2 x_\omega^2}{1-\mu_1 x_\omega}. 
\]
If it is of type $A_2^{sc}$ we have $\RS=\{\pm \al_1,\pm \al_2, \pm (\al_1 +\al_2)\}$, $\cl=\langle \omega_1,\omega_2 \rangle$ with simple roots $\al_1=2\omega_1-\omega_2$, $\al_2=2\omega_2-\omega_1$ and $x_{\al_1}=\tfrac{2x_1-\mu_1x_1^2 -x_2 - \mu_2 x_1^2 x_2}{1+\mu_2 x_1^2-\mu_1x_2-2\mu_2x_1x_2}$,
\[
C_{\al_2}(x_1)=\mu_1 x_1,\quad C_{\al_1}(x_1)=\mu_1x_1-\tfrac{1 + \mu_2x_1^2 -\mu_1x_2-2\mu_2x_1x_2}{1-\mu_1x_1 - \mu_2x_1x_2},
\]
where $x_1:=x_{\omega_1}$ and $x_2:=x_{\omega_2}$. 
\end{example}

According to \cite[\S3]{CPZ} the operators $\Dem_\al$ satisfy the twisted Leibniz rule
\begin{equation} \label{eq:diffde} 
\Dem_\al(xy)=\Dem_\al(x)y+s_\al(x)\Dem_\al(y),\quad x,y\in S, 
\end{equation}
\ie $\Dem_\al$ is a twisted derivation. Moreover, they are $S^{W_\al}$-linear, where $W_\al=\{e,s_\al\}$, and 
\begin{equation}\label{eq:invde}
s_\al(x)=x\; \text{ if and only if }\; \Dem_\al(x)=0.
\end{equation}

\begin{rem}
Properties \eqref{eq:diffde} and \eqref{eq:invde} suggest that the Demazure operators can be effectively studied using the theory of twisted derivations and the invariant theory of $W$. On the other hand, push-pull operators do not satisfy properties \eqref{eq:diffde} and \eqref{eq:invde} but according to \cite[Theorem~12.4]{CPZ} they correspond to the push-pull maps between flag varieties and, hence, are of geometric origin. 
\end{rem}

For the $i$-th simple root $\al_i$, let  $\Dem_i:=\Dem_{\al_i}$ and $s_i:=s_{\al_i}$. Given a non-empty sequence $I=(i_1,\ldots,i_m)$ with $i_j\in \{1,\ldots, n\}$ define
\[
\Dem_I:=\Dem_{i_1}\circ\cdots \circ \Dem_{i_m}\text{ and }C_I:=C_{i_1}\circ\cdots \circ C_{i_m}.
\]
We say that a sequence $I$ is \emph{reduced} in $W$ if $s_{i_1}s_{i_2}\ldots s_{i_m}$ is a reduced expression of the element $w=s_{i_1}s_{i_2}\ldots s_{i_m}$ in $W$, \ie it is of minimal length among such decompositions of $w$. In this case we also say that $I$ is a \emph{reduced sequence} for $w$ of length $\ell(w)$. For the neutral element $e$ of $W$, we set $I_e=\emptyset$ and $\Dem_\emptyset=C_\emptyset=\id_S$.

\medskip

\begin{rem}
It is well-known that for a nontrivial root datum the composites $\Dem_{I_w}$ and $C_{I_w}$ are independent of the choice of a reduced sequence $I_w$ of $w\in W$ if and only if $F$ is of the form $F(x,y)=x+y+\beta xy$, $\beta \in R$. The ``if'' part of the statement is due to Demazure \cite[Th. 1]{Dem73} and the ``only if'' part is due to Bressler-Evens \cite[Theorem~3.7]{BE90}. So for such $F$ we can define $\Dem_w:=\Dem_{I_w}$ and $C_w:=C_{I_w}$ for each $w\in W$.

The operators $\Dem_w$ and $C_w$ play a crucial role in the Schubert calculus and computations of the singular cohomology ($F=F_a$) and the $K$-theory ($F=F_m$) rings of flag varieties.

For a general $F$  (e.g. for $F=F_h$) the situation becomes much more intricate as we have to rely on choices of reduced decomposition $I_w$.
\end{rem}

Let us now prove a Euclid type lemma for later use.

\begin{lem} \label{lem:regularsum}
If $f \in x R\lbr x \rbr$ is regular in $R\lbr x \rbr$ and $g \in yR\lbr y \rbr$, then $f(x)\fplus g(y)$ is regular in $R\lbr x,y\rbr$.
\end{lem}
\begin{proof}
Consider $f\fplus g$ in $R\lbr x,y\rbr = (R\lbr x \rbr) \lbr y \rbr$ and note that its degree $0$ coefficient (in $R\lbr x \rbr$) is $f$ and is regular by assumption, so it is regular by \cite[Lemma 12.3.(a)]{CZZ}.
\end{proof}

{\small
\newcommand{\doublecell}[2][c]{\begin{tabular}[#1]{@{}c@{}}#2\end{tabular}}
\newcommand{\doublecelltop}[2][t]{\begin{tabular}[#1]{@{}c@{}}#2\end{tabular}}
\begin{table}
\begin{tabular}{l||c|c|c|c|c|c|c|c|c}
Type & \doublecelltop{$A_l$ \\ $(l\geq 2)$} & \doublecelltop{$B_l$ \\ $(l\geq 3)$} & \doublecelltop{$C_l$ \\ $(l\ge 2)$} & \doublecelltop{$D_l$ \\ $(l\geq 4)$} & $G_2$ & $F_4$ & $E_6$ & $E_7$ & $E_8$ \\
adjoint & $\emptyset$ & $2\cdot_F$ & $2\cdot_F$ & $\emptyset$ & \doublecell{$2\cdot_F$ \\ and $3\cdot_F$} & $2\cdot_F$ & $\emptyset$ & \doublecell{$2\cdot_F$ \\ or $3\cdot_F$} & \doublecell{$2\cdot_F$ \\ or $3\cdot_F$} \\
non adjoint & $|\cl/\cl_r|$ & $2$ & $2 \in R^\times$ & $2$ & - & - & $3$ & $2$ & - \\  
\end{tabular}
\caption{Integers and formal integers assumed to be regular in $R$ or $R\lbr x \rbr$ in Lemma \ref{lem:div}. In the simply connected $C_2$ case, we require $2$ invertible in $R$.}\label{table:regprimes} 
\end{table}
}

\begin{lem} \label{lem:div}
For each irreducible component of the root datum, assume that the corresponding integers or formal integers listed in Table \ref{table:regprimes} are regular in $R$ or $R\lbr x \rbr$ (and that $2$ is invertible for $C_l^{sc}$). In particular, $S$ is $\RS$-regular. Then $x_\alpha|x_\beta x'$ implies that $x_\alpha| x'$ for any two positive roots $\alpha\neq \beta$ and for any $x' \in S$.
\end{lem}
\noindent (For example, in adjoint type $E_7$ we require that either $2\cdot_F x$ or $3\cdot_F x$ is regular in $R\lbr x \rbr$, and in simply connected type $E_7$, we require that $2$ is regular in $R$.)

\begin{proof}[Proof of Lemma \ref{lem:div}]
It is equivalent to show that $x_\beta$ is regular in $S/(x_\alpha)$. 

If $\alpha$ and $\beta$ belong to different irreducible components, we can complete $\alpha$ and $\beta$ into bases of the lattices of their respective components by \cite[Lemma 2.1]{CZZ}, and then complete the union of the two sets into a basis of $\cl$. By \cite[Cor. 2.13]{CPZ}, it gives an isomorphism $S \simeq R\lbr x_1,\cdots,x_l \rbr$ sending $x_\alpha$ to $x_1$ and $x_\beta$ to $x_2$, so the conclusion is obvious in this case.

If $\alpha$ and $\beta$ belong to the same irreducible component, we can assume that the root datum is irreducible. 

\noindent \emph{Adjoint case.}
Complete $\alpha$ to a basis $(\alpha_i)_{1\leq i \leq l}$ of simple roots of $\RS$ and express $\beta = \sum_i n_i \alpha_i$. Still by \cite[Cor. 2.13]{CPZ}, this yields an isomorphism $S \simeq R \lbr x_1,\ldots,x_l \rbr$, sending $x_\alpha$ to $x_1$ and $x_\beta$ to $(n_1 \cdot_F x_1)\fplus \cdots \fplus (n_l \cdot_F x_l)$. A repeated application of Lemma \ref{lem:regularsum} shows that $x_\beta$ is regular provided $n_i\cdot_F x$ is regular in $R\lbr x \rbr$ for at least one $i \neq 1$. Using Planche I to IX in \cite{Bo68} giving coefficients of positive roots decomposed on simple ones, one checks for every type that it is always the case under the assumptions. For example, in the $E_6$ case, there are always two $1$'s in any decomposition (except if the root is simple), hence the absence of any requirement. In the $E_7$ case, the same is true except for the longest root, in which there is a $1$, a $2$ and a $3$, hence the requirement that $2\cdot_F x$ or $3\cdot_F x$ is regular in $R\lbr x \rbr$. All other cases are as easy and left to the reader. 

\medskip

\noindent \emph{Non adjoint case.}
By \cite[Lemma 1.2]{CZZ}, the natural morphism $\FGR{R}{\cl_r}{F}\to \FGR{R}{\cl}{F}$ induced by the inclusion of the root lattice $\cl_r \subset \cl$ is injective. Furthermore, it becomes an isomorphism if  $q=|\cl/\cl_r|$ is invertible in $R$.

Since $\alpha$ can be completed as a basis of $\cl$ or as a basis of $\cl_r$, both $\FGR{R}{\cl_r}{F}/x_\alpha$ and $\FGR{R}{\cl}{F}/x_\alpha$ are isomorphic to power series ring (in one less variable) and therefore respectively inject in $\FGR{R[\tfrac{1}{q}]}{\cl_r}{F}/x_\alpha$ and $\FGR{R[\tfrac{1}{q}]}{\cl_r}{F}/x_\alpha$, which are isomorphic. By the adjoint case, $x_\beta$ is regular in the latter, and thus in its subring $S/x_\alpha=\FGR{R}{\cl}{F}/x_\alpha$.
\end{proof}

\begin{rem}
Since $n\cdot_F x$ is regular in $R\lbr x \rbr$ if $n$ is regular in $R$, the conclusion of Lemma \ref{lem:div} holds when formal integers are replaced by usual integers in $R$ in the adjoint case. But more cases are covered. For example, if the formal group law is the multiplicative one $x + y -xy$, then one can show that $2\cdot_F x$ is regular in $R\lbr x \rbr$ for any noetherian ring $R$ (exercise: consider the ideal generated by the coefficients of a power series annihilating $2 \cdot_F x$), and in particular if $R=\Z[a,b]/(2a,3b)$, in which neither $3$ nor $2$ are regular, but Lemma \ref{lem:div} will still apply to all adjoint types. 
\end{rem}


\section{Two bases of the formal twisted group algebra}\label{sec:formalaffineDemazure}

We now recall definitions and basic properties of the formal twisted group algebra $\QW$, Demazure elements $X_\al$ and push-pull elements $Y_\al$, following \cite{HMSZ} and \cite{CZZ}. For a chosen set of reduced sequences $\{I_w\}_{w\in W}$ we introduce two $Q$-bases $\{X_{I_w}\}_{w\in W}$ and $\{Y_{I_w}\}_{w\in W}$ of $\QW$ and describe transformation matrices $(a^X_{v,w})$ and $(a^Y_{v,w})$ with respect to the canonical basis $\{\de_w\}_{w\in W}$ of $\QW$. 

\medskip

Let $\SW$ be the \emph{twisted group algebra} of $S$ and the group ring $R[W]$, \ie $\SW=S\otimes_R R[W]$ as an $R$-module and the multiplication is defined by
\begin{equation} \label{eq:product}
(x\otimes \de_w)(x'\otimes \de_{w'})=xw(x')\otimes\de_{ww'}, \quad x,x'\in S,\quad w,w'\in W,
\end{equation}
where $\de_w$ is the canonical element corresponding to $w$ in $R[W]$. The algebra $\SW$ is a free $S$-module with basis $\{1\otimes\de_w\}_{w\in W}$.  Note that $\SW$ is not an $S$-algebra since the embedding $S\hookrightarrow \SW$, $x\mapsto x\ot \de_e$ is not central.

\medskip

Since the formal group algebra $S$ is $\RS$-regular, it embeds into the localization $Q=S[\tfrac{1}{x_\al}\mid\al\in\RS]$. Let $\QW$ be the $Q$-module obtained by localizing the $S$-module $\SW$, \ie $\QW=Q \otimes_S S_W$. The product on $\SW$ extends to $\QW$ using the same formula \eqref{eq:product} on basis elements ($x$ and $x'$ are now in $Q$).

\medskip

Inside $\QW$, we use the notation $q :=q \ot \de_e$ and $\de_w:=1\ot \de_w$, $1:=\de_e$ and $\de_\al:=\de_{s_\al}$ for a root $\al\in \RS$. Thus $q\de_w=q \ot \de_w$ and $\de_w q=w(q)\ot \de_w$. By definition, $\{\de_w\}_{w\in W}$ is a basis of $\QW$ as a left $Q$-module, and $\SW$ injects into $\QW$ via $\de_w\mapsto \de_w$.

\medskip

For each $\al\in \RS$ we define the following elements of $\QW$ (corresponding to the operators $\Dem_\al$ and $C_\al$, respectively, by the action of \eqref{eq:leftactQW}):
\[
X_\al:=\tfrac{1}{x_\al}-\tfrac{1}{x_\al}\de_\al, \quad Y_\al:=\kp_\al-X_\al=\tfrac{1}{x_{-\al}}+\tfrac{1}{x_\al}\de_\al
\]
called the \emph{Demazure elements} and the \emph{push-pull elements}, respectively.

\medskip

Direct computations show that for each $\al\in \RS$ we have
\begin{equation}\label{eq:X_and_Y} 
X_\al^2=\kp_\al X_\al=X_\al \kp_\al\quad\text{and}\quad Y_\al^2=\kp_\al Y_\al=Y_\al\kp_\al,
\end{equation}
\[
X_\al q=s_\al(q)X_\al+\Dem_\al(q)\quad \text{and} \quad Y_\al q=s_\al(q) Y_\al+\Dem_{-\al}(q),\quad q\in Q,
\]
\[
X_\al Y_\al=Y_\al X_\al=0.
\]

\medskip

We set $\de_i:=\de_{s_i}$, $X_i:=X_{\al_i}$ and $Y_i:=Y_{\al_i}$ for the $i$-th simple root $\al_i$. Given a sequence $I=(i_1,i_2,\ldots,i_m)$ with $i_j\in \{1,\ldots,n\}$, the product $X_{i_1} X_{i_2} \ldots X_{i_m}$ is denoted by $X_I$ and the product $Y_{i_1} Y_{i_2} \ldots Y_{i_m}$ by $Y_I$. We set $X_\emptyset=Y_\emptyset=1$.

\medskip

By \cite[Ch.~VI, \S1, No~6, Cor.~2]{Bo68} if $v\in W$ has a reduced decomposition $v=s_{i_1}s_{i_2}\cdots s_{i_m}$, then 
\begin{equation}\label{eq:root}
v\RS^-\cap \RS^+=\{\al_{i_1},s_{i_1}(\al_{i_2}),\ldots,s_{i_1}s_{i_2}\cdots s_{i_{m-1}}(\al_{i_m})\}.
\end{equation}
We define 
\[
x_v:=\prod_{\al\in v\RS^-\cap \RS^+}x_\al.
\]
In particular,   $x_{w_0}=\prod_{\al\in \RS^+}x_\al$ if $w_0$ is the longest element of $W$.

\begin{lem} \label{lem:Sigmaw} We have
\begin{enumerate}
\item $s_\al\RS^-\cap \RS^+=\{\al\}$ and $x_{s_\al}=x_{\al}$;
\item \label{item:vsi} if $\ell(vs_i) =\ell(v)+1$, then 
\[
vs_i\RS^- \cap \RS^+ = (v \RS^- \cap \RS^+ )\sqcup \{v(\al_i)\}\text{ and }x_{vs_i}=x_vx_{v(\al_i)};
\]
\item \label{item:siSigma}if $\ell(s_iv)=\ell(v)+1$, then  
\[
s_iv \RS^- \cap \RS^+ = s_i(v\RS^- \cap \RS^+) \sqcup \{\al_i\}\text{ and }x_{s_iv}=s_i(x_v)x_{\al_i};
\]
\item \label{item:Sigmauv} if $w=uv$ and $\ell(w)=\ell(u)+\ell(v)$, then
\[
w\RS^- \cap \RS^+ =(u \RS^- \cap \RS^+) \sqcup u(v\RS^- \cap \RS^+)\text{ and }x_w=x_uu(x_v);\]
\item \label{item:cw0Sigma}
for any $v\in W$, $\frac{v(x_{w_0})}{x_{w_0}}$ is invertible in $S$.
\end{enumerate}
\end{lem}

\begin{proof} Items (a)-(d) follow immediately from the definition. 
As for (e) we have 
\begin{align*}
v\RS^+ &=(v\RS^+\cap \RS^-)\sqcup (v\RS^+\cap \RS^+)=\left(-(v\RS^-\cap \RS^+)\right)\sqcup (v\RS^+\cap \RS^+) \text{ and} \\
\RS^+ &=v\RS\cap \RS^+=(v\RS^-\cap \RS^+)\sqcup(v\RS^+\cup \RS^+), \text{ therefore,}
\end{align*}
\[
\tfrac{v(x_{w_0})}{x_{w_0}}=\tfrac{\prod_{\al\in v\RS^+}x_\al}{\prod_{\al\in \RS^+}x_\al}=\prod_{\al\in v\RS^-\cap\RS^+ }\tfrac{x_{-\al}}{x_\al},
\]
which is invertible in $S$ since so is $\frac{x_{-\al}}{x_\al}$.
\end{proof}
\begin{lem}\label{lem:Xdelta} 
Let $I_v$ be a reduced sequence for an element $v\in W$. 

Then $X_{I_v}=\sum_{w\le v}a^X_{v,w}\de_w$ for some $a^X_{v,w}\in Q$, where the sum is taken over all elements of $W$ less or equal to $v$ with respect to the Bruhat order and $a^X_{v,v}=(-1)^{\ell(v)}\frac{1}{x_v}$. Moreover, we have $\de_v=\sum_{w\le v}b^X_{v,w}X_{I_w}$ for some $b^X_{v,w}\in S$ such that  $b^X_{v,e}=1$ and $b^X_{v,v}=(-1)^{\ell(v)}x_v$.
\end{lem}

\begin{proof}
It follows from \cite[Lemma~5.4, Corollary~5.6]{CZZ} and the fact that $\de_\al=1-x_\al X_\al$.
\end{proof}

Similarly, for $Y$'s we have

\begin{lem}\label{lem:Ydelta}
Let $I_v$ be a reduced sequence for an element $v\in W$. 

Then $Y_{I_v}=\sum_{w\le v}a^Y_{v,w}\de_w$ for some $a^Y_{v,w}\in Q$ and $a^Y_{v,v}=\frac{1}{x_v}$. Moreover, we have $\de_v=\sum_{w\le v}b^Y_{v,w}Y_{I_w}$ for some $b^Y_{v,w}\in S$ and $b^Y_{v,v}=x_v$.
\end{lem}

\begin{proof} 
We follow the proof of \cite[Lemma~5.4]{CZZ} replacing $X$ by $Y$. By induction we have
\[
Y_{I_v}=(\tfrac{1}{x_{-\beta}}+\tfrac{1}{x_{\beta}}\de_{\beta})\sum_{w\le v'}a_{v',w}^Y \de_w=\tfrac{1}{x_\beta}s_\beta(a_{v',v'}^Y)\de_v +\sum_{w< v}a_{v,w}^Y\de_w,
\]
where $I_v=(i_1,\ldots,i_m)$ is a reduced sequence of $v$, $\beta=\al_{i_1}$ and $v'=s_\beta v$. This implies the formulas for $Y_{I_v}$ and for $a^Y_{v,v}$. Remaining statements involving $b_{v,w}^Y$ follow by the same arguments as in the proof of \cite[Corollary~5.6]{CZZ} using the fact that $\de_\al=x_\al Y_\al-\tfrac{x_\al}{x_{-\al}}$ and $\tfrac{x_\al}{x_{-\al}}\in S^\times$.
\end{proof}

As in the proof of \cite[Corollary~5.6]{CZZ}, Lemmas~\ref{lem:Xdelta} and~\ref{lem:Ydelta} immediately imply:
\begin{cor}\label{cor:basis}
The family $\{X_{I_v}\}_{v\in W}$ (resp. $\{Y_{I_v}\}_{v\in W}$) is a basis of $\QW$ as a left or as a right $Q$-module. 
\end{cor}

\begin{example} 
For the root data $A_1^{ad}$ or $A_1^{sc}$ and the formal group law $F_h$ we have $x_\Pi=x_{-\al}$ and
\[
(a_{v,w}^Y)_{v,w\in W}=\begin{pmatrix} 1 & 0 \\ \mu_1-\tfrac{1}{x_{\al}} & \tfrac{1}{x_\al} \end{pmatrix},
\]
where the first row and column correspond to $e\in W$ and the second to $s_\al\in W$.
\end{example}


\section{The Weyl and the Hecke actions}\label{sec:weylhecke}

In the present section we recall several basic facts concerning the $Q$-linear dual $\QWd$ following \cite{HMSZ} and \cite{CZZ}. We introduce a left $\QW$-action `$\act$' on $\QWd$. The latter induces an action of the Weyl group $W$ on $\QWd$ (\emph{the Weyl-action}) and the action by means of $X_\al$ and $Y_\al$ on $\QWd$ (\emph{the Hecke-action}). These two actions will play an important role in the sequel.

\medskip

Let $\QWd:=\Hom_Q(\QW,Q)$ denote the $Q$-linear dual of the left $Q$-module $\QW$.  By definition, $\QWd$ is a left $Q$-module via $(qf)(z):=qf(z)$ for any $z\in \QW$, $f\in \QWd$ and $q\in Q$.  Moreover, there is a $Q$-basis $\{f_w\}_{w\in W}$ of $\QWd$ dual to the canonical basis  $\{\de_w\}_{w\in W}$ defined by $f_w(\de_v):=\Kr_{w,v}$ (the Kronecker symbol) for $w,v\in W$.

\begin{defi}
We define a left action of $\QW$ on $\QWd$ as follows:
\[
(z\act f)(z'):=f(z'z), \quad z,z'\in \QW,\; f\in \QWd.
\]
\end{defi}

By definition, this action is left $Q$-linear, \ie $z\act (qf)=q(z\act f)$ and it induces a different left $Q$-module structure on $\QWd$ via the embedding $q\mapsto q\de_e$, \ie  
\[
(q\act f)(z):=f(zq).
\]
It also induces a $Q$-linear action of $W$ on $\QWd$ via $w(f):=\de_w\act f$.

\begin{lem}\label{lem:bulletactprop} 
We have $q\act f_w=w(q)f_w$ and $w(f_v)=f_{vw^{-1}}$ for any $q\in Q$ and $w,v\in W$. 
\end{lem}

\begin{proof} 
We have $(q\act f_w)(\de_v)=f_w(v(q)\de_v)=v(q)\Kr_{w,v}$ which shows that $q\act f_v=v(q)f_v$. For the second equality, we have $[w(f_v)](\de_u)=f_v(\de_u\de_w)=\Kr_{v,uw}$, so $w(f_v)=f_{vw^{-1}}$.
\end{proof}

There is a coproduct on the twisted group algebra $\SW$ that extends to $\QW$ defined by \cite[Def.~8.9]{CZZ}:
\[
\QWcopr: \QW\to \QW\ot_Q \QW, \quad q\de_w\mapsto q\de_w\ot \de_w.
\]
Here $\ot_Q$ is the tensor product of left $Q$-modules. It is cocommutative with  co-unit  $\QWcoun: \QW\to Q$, $q\de_w\mapsto q$ \cite[Prop.~8.10]{CZZ}. The coproduct structure on $\QW$ induces a product structure on $\QWd$, which is $Q$-bilinear for the natural action of $Q$ on $\QWd$ (not the one using $\act$). In terms of the basis $\{f_w\}_{w\in W}$ this product is given by component-wise multiplication: 
\begin{equation}\label{eq:fvprod}
(\sum_{v\in W}q_vf_v)(\sum_{w\in W}q'_wf_w)=\sum_{w\in W}q_wq'_wf_w, \quad q_w,q_w'\in Q.
\end{equation}
In other words, if we identify the dual $\QWd$ with the $Q$-module of maps $\Hom(W,Q)$ via
\[
\QWd \to \Hom(W,Q),\quad f\mapsto f',\quad f'(w):=f(\de_w),
\]
then the product is the classical multiplication of ring-valued functions. 

\medskip

The multiplicative identity $\unit$ of this product corresponds to the counit $\QWcoun$ and equals $\unit=\sum_{w\in W}f_w$. We also have 
\begin{equation}\label{eq:qcoprod}
q\act (ff')=(q\act f)f'=f(q\act f')\quad \text{ for }q\in Q \text{ and }f,f'\in \QWd.
\end{equation}

\begin{lem}\label{lem:actautom}
For any $\al\in \RS$ and $f,f'\in \QWd$ we have $s_\al(ff')=s_\al(f)s_\al(f')$, \ie the Weyl group $W$ acts on the algebra $\QWd$ by $Q$-linear automorphisms.
\end{lem}

\begin{proof}
By $Q$-linearity of the action of $W$ and of the product, it suffices to check the formula on basis elements $f=f_w$ and $f'=f_v$, for which it is straightforward. 
\end{proof}

Observe that the ring $Q$ can be viewed as a left $\QW$-module via the following action:
\begin{equation}\label{eq:leftactQW}
(q\de_w)\cdot q':=qw(q'), \quad q,q'\in Q, ~w\in W.
\end{equation}
Then by definition we have
\begin{equation}\label{eq:actone}
(q\act \unit)(z)=z\cdot q,\quad z\in \QW.
\end{equation}

\begin{defi}
For  $\al\in \RS$ we define two $Q$-linear operators on $\QWd$ by 
\[
A_\al(f):=Y_\al \act f \quad\text{ and }\quad B_\al(f):=X_\al \act f,\quad f\in \QWd.
\]
An action by means of $A_\al$ or $B_\al$ will be called a \emph{Hecke-action} on $\QWd$.
\end{defi} 

\begin{rem}
If $F=F_m$ (resp. $F=F_a$) one obtains actions introduced by Kostant--Kumar in \cite[$I_{18}$]{KK90} (resp. in \cite[$I_{51}$]{KK86}). 
\end{rem}

As in \eqref{eq:diffde} and \eqref{eq:invde} we have
\begin{equation}\label{eq:propBa} 
B_\al(ff')=B_\al(f)f'+s_\al(f)B_\al(f')\text{ and }B_\al\circ s_\al=-B_{\al},\; \text { for } f,f'\in \QWd, 
\end{equation}
\begin{equation}\label{eq:invBa}
B_\al(f)=0\;\text{ if and only if }\; f\in (\QWd)^{W_\al}.
\end{equation}
Indeed, using \eqref{eq:qcoprod} and Lemma~\ref{lem:actautom} we obtain
\begin{eqnarray*}
B_\al(f)f' +s_\al(f)B_\al(f')&=&[\tfrac{1}{x_\al}(1-\de_\al)\act f]f'+s_\al(f)[\tfrac{1}{x_\al}(1-\de_\al)\act f']\\
&=& [\tfrac{1}{x_\al}\act (f-s_\al(f))]f'+s_\al(f)[\tfrac{1}{x_\al}\act (f'-s_\al(f'))]\\
&=& \tfrac{1}{x_\al}\act (ff' - s_\al(f)s_\al(f'))=B_\al(ff')
\end{eqnarray*}
and $B_\al(s_\al(f))=\tfrac{1}{x_\al}(1-\de_\al)\act s_\al(f)=\tfrac{1}{x_\al}\act (s_\al(f)-f)=-B_\al(f)$. As for~\eqref{eq:invBa} we have $0=B_\al(f)=X_\al\act f=\tfrac{1}{x_\al}\act [(1-\de_\al)\act f]$ which is equivalent to $f=s_\al(f)$.

\medskip

And as in \eqref{eq:X_and_Y}, we obtain
\begin{equation} 
A_\al^{\circ 2}(f) =\kp_\al\act  A_\al(f)=A_\al(\kp_\al\act f),\quad B_\al^{\circ 2}(f) = \kp_\al \act B_\al(f) =B_\al (\kp_\al\act f),
\end{equation}
\[
A_\al \circ B_\al = B_\al \circ A_\al=0.
\]

We set $A_i=A_{\al_i}$ and $B_i:=B_{\al_i}$ for the $i$-th simple root $\al_i$. We set $A_I=A_{i_1}\circ \ldots \circ A_{i_m}$ and $B_I=B_{i_1}\circ \ldots \circ B_{i_m}$ for a non-empty sequence $I=(i_1,\ldots, i_m)$ with $i_j\in \{1,\ldots,n\}$ and $A_\emptyset=B_\emptyset=\id$. The operators $A_I$ and $B_I$ are key ingredients in the proof that the natural pairing of Theorem~\ref{theo:bilform} on the dual of the formal affine Demazure algebra is non-degenerate.


\section{Push-pull operators and elements} \label{sec:pushpull}

Let us now introduce and study a key notion of the present paper, the notion of push-pull operators (resp. elements) on $Q$ (resp. in $\QW$) with respect to given coset representatives in parabolic quotients of the Weyl group.

\medskip

Let $(\RS,\cl)$ be a root datum with a chosen set of simple roots $\Pi$. Let $\Xi\subseteq \Pi$ and let $W_\Xi$ denote the subgroup of the Weyl group $W$ of the root datum generated by simple reflections $s_\al$, $\al\in \Xi$. We thus have $W_\emptyset=\{e\}$ and $W_\Pi=W$.  Let $\RS_\Xi:=\{\al \in \RS \mid s_\al \in W_\Xi\}$ and let $\RS_\Xi^+:=\RS_\Xi \cap \RS^+$, $\RS_\Xi^-:=\RS_\Xi \cap \RS^-$ be subsets of positive and negative roots respectively. 

\medskip

Given subsets $\Xi' \subseteq \Xi$ of $\Pi$, let $\RS_{\Xi/\Xi'}^+:=\RS_{\Xi}^+\setminus \RS_{\Xi'}^+$ and $\RS_{\Xi/\Xi'}^-:=\RS_{\Xi}^-\setminus \RS_{\Xi'}^-$. We define 
\[
x_{\Xi/\Xi'}:=\prod_{\al \in \RS_{\Xi/\Xi'}^-} x_\al\quad \text{ and set }x_\Xi:=x_{\Xi/\emptyset}.
\]
In particular, $x_\Pi=\prod_{\al\in \RS^-}x_\al=w_0(x_{w_0})$.

\begin{lem}\label{lem:PQroot} 
Given subsets $\Xi'\subseteq \Xi$ of $\Pi$ we have
\[
v(\RS_{\Xi/\Xi'}^-)=\RS_{\Xi/\Xi'}^-\text{ and }v(\RS_{\Xi/\Xi'}^+)=\RS_{\Xi/\Xi'}^+ \text{ for any }v\in W_{\Xi'}.
\]
\end{lem}

\begin{proof} 
We prove the first statement only, the second one can be proven similarly.
Since $v$ acts faithfully on $\RS_{\Xi}$, it suffices to show that for any $\al\in \RS_{\Xi/\Xi'}^-$, the root $\beta:=v(\al)\not\in\RS_{\Xi'}$ and is negative. Indeed, if $\beta\in \RS_{\Xi'}$, then so is $\al=v^{-1}(\beta)$ (as $v^{-1}\in W_{\Xi'}$), which is impossible.  On the other hand,  if $\beta$ is positive, then  
\[
\beta=v(\al)\in v\RS_{\Xi}^-\cap \RS_{\Xi}^+=v\RS_{\Xi'}^-\cap \RS_{\Xi'}^+,
\]
where the latter equality follows from \eqref{eq:root} and the fact that $v\in W_{\Xi'}$. So $\al=v^{-1}(\beta)\in \RS_{\Xi'}$, a contradiction.
\end{proof}

\begin{cor} \label{cor:xXifixed}
For any $v \in W_{\Xi'}$, we have $v(x_{\Xi/\Xi'})=x_{\Xi/\Xi'}$.
\end{cor}

\begin{defi} \label{defi:CY}
Given a set of left coset representatives $W_{\Xi/\Xi'}$ of
$W_{\Xi}/W_{\Xi'}$ we define a \emph{push-pull operator} on $Q$ with respect to $W_{\Xi/\Xi'}$ by
\[
C_{\Xi/\Xi'}(q):=\sum_{w\in W_{\Xi/\Xi'}}w\big(\tfrac{q}{x_{\Xi/\Xi'}}\big), ~q\in Q,
\]
and a \emph{push-pull element} with respect to $W_{\Xi/\Xi'}$ by
\[
Y_{\Xi/\Xi'}:=\big(\sum_{w\in W_{\Xi/\Xi'}} \de_w\big)\tfrac{1}{x_{\Xi/\Xi'}}.
\]
We set $C_{\Xi}:=C_{\Xi/\emptyset}$ and  $Y_{\Xi}:=Y_{\Xi/\emptyset}$ (so they do not depend on the choice of $W_{\Xi/\emptyset}=W_\Xi$ in these two special cases). 
\end{defi}

By definition, we have $C_{\Xi/\Xi'}(q)=Y_{\Xi/\Xi'}\cdot q$, where $Y_{\Xi/\Xi'}$ acts on $q\in Q$ by \eqref{eq:leftactQW}. Also in the trivial case where $\Xi=\Xi'$, we have $x_{\Xi/\Xi}=1$, while $C_{\Xi/\Xi}=\id_Q$ and $Y_{\Xi/\Xi}=1$ if we choose $e$ as representative of the only coset. Observe that for $\Xi=\{\al_i\}$ we have $W_\Xi=\{e,s_i\}$ and $C_\Xi=C_i$ (resp. $Y_\Xi=Y_i$) is the push-pull operator (resp. element) introduced before and preserves $S$. 

\begin{example}
For the formal group law $F_h$ and the root datum $A_2$, we have $x_\Pi=x_{-\al_1}x_{-\al_2}x_{-\al_1-\al_2}$ and
\[
C_\Pi(1)=\sum_{w\in W} w(\tfrac{1}{x_\Pi})=\mu_1(\tfrac{1}{x_{-\al_2}x_{-\al_1-\al_2}} + \tfrac{1}{x_{-\al_1}x_{\al_2}}+\tfrac{1}{x_{\al_1}x_{\al_1+\al_2}})=\mu_1^3+\mu_1\mu_2.
\]
\end{example}

\begin{lem} \label{lem:Cinvtoinv}
The operator $C_{\Xi/\Xi'}$ restricted to $Q^{W_{\Xi'}}$ is independent of the choices of representatives $W_{\Xi/\Xi'}$ and it maps $Q^{W_{\Xi'}}$ to $Q^{W_\Xi}$.
\end{lem}

\begin{proof}
The independence follows, since $\tfrac{1}{x_{\Xi/\Xi'}}\in Q^{W_{\Xi'}}$ by Corollary~\ref{cor:xXifixed}. The second part follows, since for any $v \in W_{\Xi}$, and for any set of coset representatives $W_{\Xi/\Xi'}$, the set $v W_{\Xi/\Xi'}$ is again a set of coset representatives. 
\end{proof}
Actually, we will see in Corollary \ref{cor:opAinv2} that the operator $C_{\Xi}$ sends $S$ to $S^{W_{\Xi}}$. 

\begin{rem}
The formula for the operator $C_\Xi$ (with $\Xi'=\emptyset$) had appeared before in related contexts, namely, in discussions around the Becker-Gottlieb
transfer for topological complex-oriented theories (see \cite[(2.1)]{BE90} and \cite[\S4.1]{GaRa}).
The definition of the element $Y_{\Xi/\Xi'}$ can be viewed as a generalized algebraic analogue of this formula.
\end{rem}

\begin{lem}[Composition rule] \label{lem:pullpushcomp} 
Given subsets $\Xi''\subseteq \Xi'\subseteq \Xi$ of $\Pi$ and given sets of representatives $W_{\Xi/\Xi'}$ and $W_{\Xi'/\Xi''}$, take $W_{\Xi/\Xi''}:=\{wv\mid w\in W_{\Xi/\Xi'},\, v\in W_{\Xi'/\Xi''}\}$ as the set of representatives of $W_\Xi/W_{\Xi''}$. Then
\[
C_{\Xi/\Xi'} \circ C_{\Xi'/\Xi''}=C_{\Xi/\Xi''}\text{ and }Y_{\Xi/\Xi'}Y_{\Xi'/\Xi''}=Y_{\Xi/\Xi''}.
\]
\end{lem}

\begin{proof} 
We prove the formula for $Y$'s, the one for $C$'s follows since $C$ acts as $Y$, and the composition of actions corresponds to multiplication. We have $Y_{\Xi/\Xi'}Y_{\Xi'/\Xi''}=$
\[
(\sum_{w\in W_{\Xi/\Xi'}}\de_w\tfrac{1}{x_{\Xi/\Xi'}})(\sum_{v\in W_{\Xi'/\Xi''}}\de_v\tfrac{1}{x_{\Xi'/\Xi''}})=\sum_{w\in W_{\Xi/\Xi'},\, v\in W_{\Xi'/\Xi''}}\de_{wv}\tfrac{1}{v^{-1}(x_{\Xi/\Xi'})x_{\Xi'/\Xi''}}.
\]
By Corollary~\ref{cor:xXifixed}, we have $v^{-1}(x_{\Xi/\Xi'})=x_{\Xi/\Xi'}$. Therefore, $v^{-1}(x_{\Xi/\Xi'})x_{\Xi'/\Xi''}=x_{\Xi/\Xi'}x_{\Xi'/\Xi''}=x_{\Xi/\Xi''}$. We conclude by definition of $ W_{\Xi/\Xi''}$.
\end{proof}

The following lemma follows from the definition of $C_{\Xi/\Xi'}$.
\begin{lem}[Projection formula] \label{lem:projformC}
We have
\[
C_{\Xi/\Xi'}(qq')=q\, C_{\Xi/\Xi'}(q')\quad \text{ for any }q\in Q^{W_\Xi}\text{ and }q'\in Q.
\]
\end{lem}

\begin{lem} \label{lem:decomppullpush} 
Given a subset $\Xi$ of $\Pi$ and $\al\in \Xi$ we have
\begin{enumerate}
\item \label{item:YY}
$Y_\Xi=Y'Y_\al=Y_\al Y''~\text{ for some }Y' \text{ and }Y''\in \QW$,
\item \label{item:YX}
$Y_\Xi X_\al=X_\al Y_\Xi=0$, $Y_\al Y_\Xi=\kp_\al Y_\Xi$ and $Y_\Xi Y_\al=Y_\Xi\kp_\al$.
\end{enumerate}
\end{lem}

\begin{proof} 
\eqref{item:YY} The first identity follows from Lemma~\ref{lem:pullpushcomp} applied to $\Xi'=\{\al\}$ (in this case $Y'=Y_{\Xi/\Xi'}$).

\medskip

For the second identity, let ${}^\al W_\Xi$ be the set of right coset representatives of $W_\al\backslash W_{\Xi}$, thus each $w\in W_\Xi$ can be written uniquely either as $w=s_\al u$ or as $w=u$ with $u\in {}^\al W_\Xi$. Then 
\begin{eqnarray*}
Y_\Xi &=& \sum_{u\in {}^\al W_\Xi}(1+\de_\al)\de_{u}\tfrac{1}{x_\Xi} =\sum_{u\in {}^\al W_\Xi}(1+\de_\al)\tfrac{1}{x_{-\al}}x_{-\al}\de_u\tfrac{1}{x_\Xi}\\
 &=& \sum_{u\in {}^\al W_\Xi}Y_\al x_{-\al}\de_u\tfrac{1}{x_\Xi}= Y_\al\sum_{u\in {}^\al W_\Xi}\de_u\tfrac{u^{-1}(x_{-\al})}{x_\Xi}.
\end{eqnarray*}

\eqref{item:YX} then follows from \eqref{item:YY} and \eqref{eq:X_and_Y}.
\end{proof}


\section{The push-pull operators on the dual}\label{sec:pushpulldual}

We now introduce and study the push-pull operators on the dual of the twisted formal group algebra $\QWd$.

\medskip

For $w\in W$, we define $f_w^\Xi:=\sum_{v\in wW_\Xi}f_{v}$. Observe that $f_w^\Xi = f_{w'}^\Xi$ if and only if $w W_\Xi = w' W_\Xi$. Consider the subring of invariants $(\QWd)^{W_\Xi}$ by means of the `$\act$'-action of $W_\Xi$ on $\QWd$ and fix a set of representatives $W_{\Pi/\Xi}$ of $W/W_\Xi$. By Lemma \ref{lem:bulletactprop}, we then have the following

\begin{lem} \label{lem:QWinvariant} 
The family $\{f_w^\Xi\}_{w\in W_{\Pi/\Xi}}$ forms a basis of $(\QWd)^{W_\Xi}$ as a left $Q$-module, and $f_w^\Xi f_v^\Xi =\Kr_{w,v} f_v^\Xi$ for any $w,v\in W_{\Pi/\Xi}$. 
\end{lem}
In other words, $\{f_w^\Xi\}_{w \in W_{\Pi/\Xi}}$ is a set of pairwise orthogonal projectors, and the direct sum of their images is $(\QWd)^{W_{\Xi}}$.

\begin{defi}
Given subsets $\Xi'\subseteq \Xi$ of $\Pi$ and a set of representatives $W_{\Xi/\Xi'}$ we define a $Q$-linear operator on $\QWd$ by
\[
A_{\Xi/\Xi'}(f):=Y_{\Xi/\Xi'}\act f,\quad f\in \QWd,
\]
and call it the \emph{push-pull operator} with respect to $W_{\Xi/\Xi'}$. It is $Q$-linear since so is the $'\act'$-action. We set $A_\Xi=A_{\Xi/\emptyset}$.
\end{defi}

Lemma~\ref{lem:pullpushcomp} immediately implies:
\begin{lem}[Composition rule]\label{lem:compositA} 
Given subsets $\Xi''\subseteq \Xi' \subseteq \Xi$ of $\Pi$ and sets of representatives $W_{\Xi/\Xi'}$ and $W_{\Xi'/\Xi''}$, let $W_{\Xi/\Xi''}=\{wv\mid w\in W_{\Xi/\Xi'},\, v\in W_{\Xi'/\Xi''}\}$, then  we have $A_{\Xi/\Xi'}\circ A_{\Xi'/\Xi''}=A_{\Xi/\Xi''}$. 
\end{lem}

\begin{lem}[Projection formula] \label{lem:projformulaA} 
We have  
\[
A_{\Xi/\Xi'}(ff')=fA_{\Xi/\Xi'}(f')\quad \text{ for any }f\in (\QWd)^{W_\Xi}\text{ and }f'\in \QWd.
\]
\end{lem}
\begin{proof} 
Using \eqref{eq:qcoprod} and Lemma \ref{lem:actautom}, we compute 
\[
\begin{split}
& A_{\Xi/\Xi'}(ff')=Y_{\Xi/\Xi'}\act (ff') = \big(\sum_{w\in W_{\Xi/\Xi'}} \hspace{-.5ex} \de_w \tfrac{1}{x_{\Xi/\Xi'}}\big) \act (ff') = \hspace{-1ex}\sum_{w\in W_{\Xi/\Xi'}} \hspace{-.5ex}\de_w \act \tfrac{1}{x_{\Xi/\Xi'}} \act (ff') \\
& = \sum_{w\in W_{\Xi/\Xi'}} \hspace{-.5ex} \de_w \act \big(f (\tfrac{1}{x_{\Xi/\Xi'}} \act  f')\big) =\sum_{w\in W_{\Xi/\Xi'}} \hspace{-.5ex}(\de_w \act f)  (\de_w \act \tfrac{1}{x_{\Xi/\Xi'}} \act  f') \\
& = f \sum_{w\in W_{\Xi/\Xi'}} \hspace{-.5ex}\de_w \act \tfrac{1}{x_{\Xi/\Xi'}} \act  f' = f  A_{\Xi/\Xi'}(f') \mbox{\qedhere}
\end{split}
\]
\end{proof}

Here is an analogue of Lemma \ref{lem:Cinvtoinv}.
\begin{lem} \label{lem:Ainvtoinv}
The operator $A_{\Xi/\Xi'}$ restricted to $(\QWd)^{W_{\Xi'}}$ is independent of the choices of representatives $W_{\Xi/\Xi'}$ and it maps $(\QWd)^{W_{\Xi'}}$ to $(\QWd)^{W_\Xi}$.
\end{lem}

\begin{proof}
Let $f \in (\QWd)^{W_{\Xi'}}$. For any $w\in W$ and  $v \in W_{\Xi'}$, by Corollary \ref{cor:xXifixed}, we have 
\[
\big(\de_{wv} \tfrac{1}{x_{\Xi/\Xi'}}\big) \act f = \big(\de_w \tfrac{1}{x_{\Xi/\Xi'}} \de_v \big) \act f = \big(\de_w \tfrac{1}{x_{\Xi/\Xi'}}\big) \act \de_v \act f = \big(\de_w \tfrac{1}{x_{\Xi/\Xi'}}\big)\act f.
\]
which proves that the action on $f$ of any factor $\de_w (\tfrac{1}{x_{\Xi/\Xi'}})$ in $Y_{\Xi/\Xi'}$ is independent of the choice of the coset representative $w$. 

Now if $v\in W_{\Xi}$, we have
\[ 
v (A_{\Xi/\Xi'}(f)) = \de_v \act Y_{\Xi/\Xi'} \act f = (\de_v Y_{\Xi/\Xi'}) \act f = A_{\Xi/\Xi'}(f),
\]
where the last equality holds since $\de_v Y_{\Xi/\Xi'}$ is again an operator $Y_{\Xi/\Xi'}$ corresponding to the set of coset representatives $v W_{\Xi/\Xi'}$  (instead of $W_{\Xi/\Xi'}$). This proves the second claim. 
\end{proof}

\begin{lem} \label{lem:APonfv} 
We have $A_{\Xi/\Xi'}(f_v)=\tfrac{1}{v(x_{\Xi/\Xi'})}\sum_{w\in W_{\Xi/\Xi'}}f_{vw^{-1}}$. In particular,  
\[
A_{\Xi/\Xi'}(f_v^{\Xi'})=\tfrac{1}{v(x_{\Xi/\Xi'})}f_v^{\Xi}, \qquad A_{\Pi/\Xi}(f_v^{\Xi})=\tfrac{1}{v(x_{\Pi/\Xi})}\unit\quad \text{and} \quad A_\Pi(v(x_\Pi) f_v)=\unit.
\]
\end{lem}

\begin{proof} 
By Lemma \ref{lem:bulletactprop} we get 
\[
A_{\Xi/\Xi'}(f_v)
= \big(\hspace{-1ex}\sum_{w\in W_{\Xi/\Xi'}}\hspace{-2ex}\de_w\tfrac{1}{x_{\Xi/\Xi'}}\big)\act f_v
= \hspace{-1ex}\sum_{w\in W_{\Xi/\Xi'}}\hspace{-2ex}\de_w\act \big(\tfrac{1}{v(x_{\Xi/\Xi'})}f_v\big)
=\tfrac{1}{v(x_{\Xi/\Xi'})}\hspace{-1ex}\sum_{w\in W_{\Xi/\Xi'}}f_{vw^{-1}}.
\]
In particular
\[
\begin{split}
A_{\Xi/\Xi'}(f_v^{\Xi'}) 
& = \sum_{w \in W_{\Xi'}}\tfrac{1}{vw(x_{\Xi/\Xi'})}\sum_{u \in W_{\Xi/\Xi'}} f_{vwu^{-1}} 
  = \tfrac{1}{v(x_{\Xi/\Xi'})}\sum_{w \in W_{\Xi'}} \sum_{u \in W_{\Xi/\Xi'}} f_{vwu^{-1}} \\
& = \tfrac{1}{v(x_{\Xi/\Xi'})}\sum_{w \in vW_{\Xi}} f_w 
  = \tfrac{1}{v(x_{\Xi/\Xi'})} f_v^{\Xi},
\end{split}
\]
where the second equality follows from Corollary \ref{cor:xXifixed}.
\end{proof}

Together with Lemma~\ref{lem:QWinvariant} we therefore obtain:
\begin{cor}
We have $A_{\Xi/\Xi'}((\QWd)^{W_{\Xi'}})=(\QWd)^{W_\Xi}$.
\end{cor}

\begin{defi}\label{defi:charmap}
We define the \emph{characteristic map} $c\colon Q \to \QWd$ by $q\mapsto q\act \unit$. 
\end{defi}
By the definition of the `$\act$' action, $c$ is an $R$-algebra homomorphism given by $c(q)=\sum_{w\in W}w(q)f_w$, that is, $c(q)\in \QWd$ is the evaluation at $q\in \QW$ via the action \eqref{eq:leftactQW} of $\QW$ on $Q$. Note that  $c$ is $\QW$-equivariant with respect to this action and the `$\act$'-action.
Indeed, $c(z\cdot q)=(z\cdot q) \act \unit=z\act (q\act \unit)=z\act c(q)$. In particular, it is $W$-equivariant.

\medskip

The following lemma provides an analogue of the push-pull formula of \cite[Theorem. 12.4]{CPZ}. 
\begin{lem} 
Given subsets $\Xi'\subseteq \Xi$ of $\Pi$, we have $A_{\Xi/\Xi'}\circ c =c \circ  C_{\Xi/\Xi'}$.
\end{lem}

\begin{proof} 
By definition, we have
\[
A_{\Xi/\Xi'}(c(q))=Y_{\Xi/\Xi'}\act c(q)=c(Y_{\Xi/\Xi'}\cdot q) = c(C_{\Xi/\Xi'}(q)). \qedhere
\]
\end{proof}


\section{Relations between bases coefficients} \label{sec:coeff}

In this section we describe relations between coefficients appearing in decompositions of various elements on the different bases of $\QW$ and of $\QWd$.

\medskip

Given a sequence $I=(i_1,\ldots,i_m)$, let $I^{\rev}:=(i_m,\ldots,i_1)$.
\begin{lem} \label{lem:CPi_I_Irev}
Given a sequence $I$ in $\{1,\ldots,n\}$, for any $x,y\in S$ and $f,f'\in \QWd$ we have
\[
C_\Pi(\Dem_I(x)y)=C_\Pi(x\Dem_{I^{\rev}}(y))\;\text{ and }\; A_\Pi(B_I(f)f')=A_\Pi(fB_{I^{\rev}}(f')).
\]
Similarly, we have 
\[
C_\Pi(C_I(x)y)=C_\Pi(xC_{I^{\rev}}(y))\;\text{ and }\; A_\Pi(A_I(f)f')=A_\Pi(fA_{I^{\rev}}(f')).
\]
\end{lem}

\begin{proof} 
By Lemma~\ref{lem:decomppullpush}.(b) we have  $Y_\Pi X_\al=0$ for any $\al\in \Pi$. 
By \eqref{eq:propBa} we obtain
\[
0=A_\Pi\big(B_\al(s_\al(f)f')\big)=A_\Pi\big(fB_\al(f')-B_\al(f)f'\big).
\]
Hence, $A_\Pi(B_\al(f)f')=A_\Pi(fB_\al(f'))$ and  $A_\Pi(B_I(f)f')=A_\Pi(fB_{I^{\rev}}(f'))$ by iteration. 

To prove the corresponding formula involving $A_I$, note that $A_\al=\kp_\al-B_\al$, so 
\[
\begin{split}
fA_\al(f')-A_\al(f)f' & = f(\kp_\al\act f'-B_\al(f'))-(\kp_\al \act f-B_\al(f))f' \\
 & \hspace{-1ex}\equalbyeq{eq:qcoprod} B_\al(f)f'-fB_\al(f')=B_\al(s_\al(f')f),
\end{split}
\]
so $A_\Pi(A_\al(f)f')=A_\Pi(fA_\al(f'))$ and again $A_\Pi(A_{I}(f)f')=A_\Pi(fA_{I^{\rev}}(f'))$ by iteration.
The formulas involving $C$ operators are obtained similarly. 
\end{proof}

\begin{cor}\label{cor:I_and_Irev} 
Let $I=(i_1,\ldots,i_m)$ be a sequence in $\{1,\ldots,n\}$. Let
\[
X_I=\sum_{v\in W}a^X_{I,v}\de_v \; \text{ and } \; X_{I^{\rev}}=\sum_{v\in W}a'^X_{I,v}\de_v \quad\text{ for some } a^X_{I,v},\; a'^X_{I,v} \in Q,
\] 
then $v(x_\Pi)\, a'^X_{I,v}=v(a^X_{I,v^{-1}})\, x_\Pi$. Similarly, let
\[
Y_{I}=\sum_{v\in W}a^Y_{I,v}\de_v \; \text{ and }\; Y_{I^{\rev}}=\sum_{v\in W}a'^Y_{I,v}\de_v\quad\text{ for some } a^Y_{I,v},\; a'^{Y}_{I,v}\in Q,
\]
then $v(x_\Pi)\, a'^Y_{I,v}=v(a^Y_{I,v^{-1}}) \,x_\Pi$.
\end{cor}
\begin{proof}
We have 
\[
\begin{split}
v(x_\Pi) A_\Pi \big( B_I(f_e)f_v \big) 
& = v(x_\Pi) A_\Pi \big( (X_I\act f_e)f_v \big) 
= v(x_\Pi) A_\Pi \big( (\sum_w w^{-1}(a^X_{I,w})f_{w^{-1}})f_v \big) \\
& = v(x_\Pi) A_\Pi \big( v(a^X_{I,v^{-1}})f_v \big)
\equalbyref{lem:APonfv} v \big( a^X_{I,v^{-1}} \big)\unit,
\end{split}
\]
and symmetrically 
\[
\begin{split}
x_\Pi A_\Pi \big( f_e B_{I^{\rev}}(f_v) \big) 
& = x_\Pi A_\Pi \bigg( f_e \sum_w a'^X_{I,w}\de_w \act f_v  \bigg) \\
& = x_\Pi A_\Pi \bigg( f_e \sum_w vw^{-1}(a'^X_{I,w} ) f_{vw^{-1}} \bigg) \\
& = x_\Pi A_\Pi(a'^X_{I,v}f_e) 
= a'^X_{I,v} \unit.
\end{split}
\]
Lemma \ref{lem:CPi_I_Irev} then yields the formula by comparing the coefficients of $X_I$ and $X_{I^{\rev}}$. The formula involving $Y_I$ is obtained similarly. 
\end{proof}

\begin{lem}\label{lem:AIrev}
For any sequence $I$, we have 
\[
A_{I^{\rev}}(x_\Pi f_e)=\sum_{v\in W}v(x_\Pi)a^Y_{I,v}f_v\quad\text{ and }\quad
B_{I^{\rev}}(x_\Pi f_e)=\sum_{v\in W}v(x_\Pi)a^X_{I,v}f_v.
\]
\end{lem}

\begin{proof} 
We prove the first formula only. The second one can be obtained using similar arguments.
Let $Y_{I^{\rev}}=\sum_{v\in W}a'^Y_{I,v}\de_v$ and $Y_I=\sum_{v\in W}a^Y_{I,v}\de_v$ as in Corollary~\ref{cor:I_and_Irev}. 
\[
\begin{split}
A_{I^{\rev}}(x_\Pi f_e) 
& = Y_{I^{\rev}}\act x_\Pi f_e 
= \sum_{v\in W} x_\Pi (a'^Y_{I,v}\de_v\act f_e) \\
& = \sum_{v\in W}x_\Pi (a'^Y_{I,v}\act f_{v^{-1}}) 
= \sum_{v\in W}x_\Pi v^{-1}(a'^Y_{I,v})f_{v^{-1}} 
= \sum_{v\in W}x_\Pi v(a'^Y_{I,v^{-1}})f_{v}.
\end{split}
\]
The formula then follows from Corollary~\ref{cor:I_and_Irev}.
\end{proof}

Let $\{X_{I_w}^*\}_{w\in W}$ and $\{Y_{I_w}^*\}_{w\in W}$ be the $Q$-linear bases of $\QWd$ dual to $\{X_{I_w}\}_{w\in W}$ and $\{Y_{I_w}\}_{w\in W}$, respectively, \ie $X_{I_w}^*(X_{I_v})=\Kr_{w,v}$ for $w,v\in W$. By Lemma~\ref{lem:Xdelta} we have $\de_v=\sum_{w\le v}b^X_{v,w}X_{I_w}=\sum_{w\le v}b^Y_{v,w}Y_{I_w}$. Therefore, by duality we have
\begin{equation}\label{eq:XYdualbasis}
X_{I_w}^*=\sum_{v\geq w}b^X_{v,w}f_v \quad\text{and}\quad Y_{I_w}^*=\sum_{v\geq w}b^Y_{v,w}f_v.
\end{equation}

\begin{lem} \label{lem:constcoeff}
We have $X_{I_e}^*=\unit$ and, therefore, $X_{I_e}^*(z)=z\cdot 1$, $z\in \QW$ (the action defined in \eqref{eq:leftactQW}). For any sequence $I$ with $\ell(I)\ge 1$, we have $X_{I_e}^*(X_I)=X_I\cdot 1=0$ and, moreover, if we express $X_I=\sum_{v\in W}q_vX_{I_v}$, then $q_e=0$. 
\end{lem}
\begin{proof}
Indeed, for each $v\in W$ we have $X_{I_e}^*(\de_v)=b^X_{v,e}=1=\unit(\de_v)$.  Therefore, $X_{I_e}^*=\unit$. The formula for $X_{I_e}^*(z)$ then follows by~\eqref{eq:actone}. Since $X_\al \cdot 1=0$, we have $X_I\cdot 1=0$. Finally, we obtain
\[
0=X_I\cdot 1=\sum_{v\in W}q_vX_{I_v}\cdot 1=q_e+\sum_{\ell(v)\ge 1} q_v X_{I_v}\cdot 1=q_e.\qedhere
\]
\end{proof}

\begin{lem}
Let $w_0$ be the longest element in $W$ of length $N$. We have
\[
A_\Pi(X_{I_{w_0}}^*)=(-1)^N\unit \quad \text{ and }\quad A_\Pi(Y_{I_{w_0}}^*)=\unit.
\]
\end{lem}

\begin{proof}
Consider the first formula. By Lemma \ref{lem:Xdelta}  $\de_v=\sum_{w\le v}b^X_{v,w}X_{I_w}$ with $b^X_{w,w}=x_w$, therefore $X_{I_w}^*=\sum_{v\ge w}b^X_{v,w}f_v$. Lemma~\ref{lem:APonfv} yields
\[
A_\Pi(X_{I_w}^*)=\sum_{v\ge w}\tfrac{b^X_{v,w}}{v(x_\Pi)}\unit.
\]
If $w=w_0$ is the longest element, then $A_\Pi(X_{I_{w_0}}^*)=\tfrac{(-1)^Nx_{w_0}}{w_0(x_\Pi)}\unit=(-1)^N\unit$ by \eqref{eq:root}.

The second formula is obtained similarly using Lemma~\ref{lem:Ydelta} instead.
\end{proof}

\begin{lem}\label{lem:commrelX} 
For any reduced sequence $I$ of an element $w$ and $q\in Q$ we have
\[
X_Iq=\sum_{v\le w} \phi_{I,v}(q) X_{I_v}\quad \text{ for some }\phi_{I,v}(q)\in Q.
\]
\end{lem}

\begin{proof} 
For any subsequence $J$ of $I$ (not necessarily reduced), we have $w(J) \leq w$ by \cite[Th. 1.1]{Deo77}. Thus, by developing all $X_i=\tfrac{1}{x_{i}}(1-\de_{i})$, moving all coefficients to the left, and then using Lemma \ref{lem:Xdelta} and transitivity of the Bruhat order,
\[
X_I q = \sum_{w\leq v} \tilde{\phi}_{I,w}(q) \de_w = \sum_{w\leq v} \phi_{I,w}(q) X_{I_w} 
\] 
for some coefficients $\tilde{\phi}_{I,w}(q)$ and $\phi_{I,w}(q) \in Q$.
\end{proof}


\section{Another basis of the $W_\Xi$-invariant subring}\label{sec:anotherbasis}

Recall that $\{f_w^\Xi\}_{w\in W_{\Pi/\Xi}}$ is a basis of the invariant subring $(\QWd)^{W_\Xi}$. In the present section we construct another basis $\{X_{I_u}^*\}_{u\in W^\Xi}$ of the subring $(\QWd)^{W_\Xi}$, which generalizes \cite[Lemma 4.34]{KK86} and \cite[Lemma 2.27]{KK90}.

\medskip

Given a subset $\Xi$ of $\Pi$ we define
\[
W^\Xi=\{w\in W\mid\ell(ws_\al)>\ell(w) \text{ for any } \al\in \Xi\}.
\]
Note that $W^\Xi$ is a set of left coset representatives of $W/W_\Xi$ such that each $w\in W^\Xi$ is the unique representative of  minimal length. 

We will extensively use the following fact \cite[\S 1.10]{Hu90}:  
\begin{equation} \label{prop:coset}
\text{For any }w\in W\text{ there exist  unique }u\in W^\Xi\text{ and }v\in W_\Xi 
\end{equation}
\[
\text{ such that }w=uv\text{ and }\ell(w)=\ell(u)+\ell(v).
\]

\begin{defi}
Let $\Xi$ be a subset of $\Pi$. We say that the reduced sequences $\{I_w\}_{w\in W}$ are {\em $\Xi$-compatible} if for each $w\in W$ and the unique factorization $w=uv$ with $u\in W^\Xi$ and $v\in W_\Xi$, $\ell(w)=\ell(u)+\ell(v)$ of \eqref{prop:coset} we have $I_w=I_u \conc I_v$, \ie $I_w$ starts with $I_u$ and ends by $I_v$. 
\end{defi}

Observe that there always exists a $\Xi$-compatible family of reduced sequences. Indeed, one could start with arbitrary reduced sequences $\{I_u\}_{u\in W^\Xi}$ and $\{I_v\}_{v\in W_\Xi}$, and complete it into a $\Xi$-compatible family $\{I_w\}_{w\in W}$ by defining $I_w$ as the concatenation $I_u\cup I_v$ for $w=uv$ with $u\in W^\Xi, v\in W_\Xi$. 

\medskip

\begin{theo} \label{th:WPinvariant} 
For any $\Xi$-compatible choice of reduced sequences $\{I_w\}_{w\in W}$, if $u\in W^\Xi$, then for any sequence $I$ in $W_{\Xi}$ of length at least $1$ (\ie $\al_i \in \Xi$ for each $i$ appearing in the sequence $I$), we have 
\[
X_{I_u}^*(zX_I)=0\text{ for all }z\in \QW.
\]   
\end{theo}
 
\begin{proof}
Since $\{X_{I_w}\}_{w\in W}$ is a basis of $\QW$, we may assume that $z=X_{I_w}$ for some $w\in W$. We decompose $X_I= \sum_{v \in W_{\Xi}} q_v X_{I_v}$ with $q_v \in Q$. By Lemma~\ref{lem:constcoeff} we may assume $v\neq e$. 

We proceed by induction on the length of $w$. If $\ell(w)=0$, we have $X_{I_w}=X_{I_e}=1$. Since $W_{\Xi}\cap W^{\Xi}=\{e\}$, for any $v\in W_\Xi$, $v\neq e$, we conclude that $X_{I_u}^*(X_{I_v})=0$.

The induction step goes as follows: Assume $\ell(w)\ge 1$. Since the sequences are $\Xi$-compatible, we have
\[
X_{I_w} X_I = X_{I_{w'}} X_{I_{v'}} X_I = X_{I_{w'}} X_{I'},\text{ where }w'\in W^{\Xi},\; v' \in W_{\Xi},\; I' \in W_{\Xi},\text{ and}
\]
$\ell(I')\geq \ell(I)\geq 1$. We can thus assume that $w \in W^{\Xi}$, so that by Lemma \ref{lem:commrelX},
\[
X_{I_w}X_I=\sum_{v\neq e} (X_{I_w} q_{v}) X_{I_{v}} = \sum_{\tilde w\leq w, v \neq e} \phi_{I_{w},\tilde w}(q_{v}) X_{I_{\tilde w}} X_{I_{v}}.
\]
Now $X_{I_u}^*(X_{I_w}X_{I_v})=X_{I_u}^*(X_{I_{wv}})=0$ since $wv$ is not a minimal coset representative: indeed, we already have $w\in W^\Xi$ and $v\neq e$. Applying $X_{I_u}^*$ to other terms in the above summation gives zero by induction. 
\end{proof}

\begin{rem} 
The proof will not work if we replace $X$'s by $Y$'s, because constant terms appear (we can not assume $v\neq e$). 
\end{rem}

\begin{cor}\label{cor:invbasisq}
For any $\Xi$-compatible choice of reduced sequences $\{I_{u}\}_{u\in W}$, the family  $\{X_{I_u}^*\}_{u\in W^\Xi}$ is a $Q$-module basis of $(\QWd)^{W_\Xi}$.
\end{cor}

\begin{proof}
 For every $\al_i\in \Xi$ we have 
\[
(\de_i \act X_{I_u}^*)(z)=X_{I_u}^*(z \de_i)=X_{I_u}^*(z (1-x_iX_i))=X_{I_u}^*(z), \quad z\in \QW,
\]
where the last equality follows from Theorem~\ref{th:WPinvariant}. Therefore, $X_{I_u}^*$ is $W_\Xi$-invariant.

Let $\sigma\in (\QWd)^{W_\Xi}$, \ie for each $\al_i\in \Xi$ we have $\sigma=s_i(\sigma)=\de_i\act \sigma$. Then
\[
\sigma(z X_i)=\sigma(z \tfrac{1}{x_{\al_i}}(1-\de_{\al_i}))=\sigma(z \tfrac{1}{x_{\al_i}})-(\de_i\act \sigma)(z\tfrac{1}{x_{\al_i}})=(\sigma -\de_i \act \sigma)(z\tfrac{1}{x_{\al_i}})=0  
\]
for any $z\in \QW$. Write $\sigma=\sum_{w\in W}x_w X_{I_w}^*$ for some $x_w\in Q$. If $w\notin W^\Xi$, then $I_w$ ends by some $i$ such that $\al_i\in \Xi$ which implies that 
\[
x_w=\sigma(X_{I_w})=\sigma(X_{I'_w} X_i)=0, 
\]
where $I'_w$ is the sequence obtained by deleting the last entry in $I_w$. So $\sigma$ is a linear combination of $\{X_{I_u}^*\}_{u\in W^\Xi}$.
\end{proof}

\begin{cor} \label{cor:inv_coeff}
If the reduced sequences $\{I_w\}_{w\in W}$ are $\Xi$-compatible, then $b^X_{wv,u}=b^X_{w,u}$ for any $v\in W_\Xi$, $u\in W^\Xi$ and $w\in W$, where $b^X_{wv,u}$ are the coefficients of Lemma~\ref{lem:Xdelta}.
\end{cor}

\begin{proof}  
From Lemma~\ref{lem:Xdelta} we have $X_{I_u}^*=\sum_{w\ge u}b^X_{w,u}f_w$. By Lemma~\ref{lem:bulletactprop} we obtain that $v(X_{I_u}^*)=\sum_{w\ge u}b^X_{w,u}f_{wv^{-1}}$ for any $v\in W_\Xi$. Since $X_{I_u}^*$ is $W_\Xi$-invariant by Corollary~\ref{cor:invbasisq} and $\{f_w\}_{w\in W}$ is a basis of $\QWd$, this implies that $b^X_{wv^{-1},u}=b^X_{w,u}$.
\end{proof}


\section{The formal Demazure algebra and the Hecke algebra}\label{sec:iwahorihecke}

In the present section we recall the definition and basic properties of the formal (affine) Demazure algebra $\DcF$ following \cite{HMSZ}, \cite{CZZ} and \cite{Zh}. 

\medskip

Following \cite{HMSZ}, we define the \emph{formal affine Demazure algebra} $\DcF$ to be the $R$-subalgebra of the twisted formal group algebra $\QW$ generated by elements of $S$ and the Demazure elements $X_i$ for all $i\in \{1,\ldots,n\}$. By \cite[Lemma 5.8]{CZZ}, $\DcF$ is also generated by $S$ and all $X_\al$ for all $\al\in \RS$. Since $\kp_\al\in S$, the algebra $\DcF$ is also generated by $Y_\al$'s and elements of $S$. Finally, since $\de_\al = 1-x_\al X_\al$, all elements $\de_w$ are in $\DcF$, and $\DcF$ is a sub-$\SW$-module of $\QW$, both on the left and on the right. 

\begin{rem}
Since $\{X_{I_w}\}_{w\in W}$ is a $Q$-basis of $\QW$, restricting the action \eqref{eq:leftactQW} of $\QW$ onto $\DcF$  we obtain an isomorphism between the algebra $\DcF$ and the $R$-subalgebra $\ED$ of $\End_R(S)$ generated by operators $\Dem_\al$ (resp. $C_\al$) for all $\al\in \RS$, and multiplications by elements from $S$. This isomorphism maps $X_\al \mapsto \Dem_\al$ and $Y_\al \mapsto C_\al$. Therefore, for any identity or statement involving elements $X_\al$ or $Y_\al$ there is an equivalent identity or statement involving operators $\Dem_\al$ or $C_\al$.
\end{rem}

According to \cite[Theorem~6.14]{HMSZ} (or \cite[7.9]{CZZ} when the ring $R$ is not necessarily a domain),
in type $A_n$, the algebra $\DcF$ is generated by the Demazure elements $X_i$, $i\in \{1,\dots,n\}$,
and multiplications by elements from $S$ subject to the following relations:
\begin{enumerate}[\indent(a)]
\item \label{item:square} $X_i^2= \kp_i X_i$ 
\item \label{item:commute} $X_iX_j=X_jX_i$ for $|i-j|>1$,
\item \label{item:braid} $X_i X_j X_i-X_jX_iX_j=\kappa_{ij}(X_j-X_i)$ for $|i-j|=1$ and
\item \label{item:scalar} $X_iq=s_i(q)X_i+\Dem_i(q)$,
\end{enumerate}

Furthermore, by \cite[Prop. 7.7]{CZZ}, for any choice of reduced decompositions $\{I_w\}_{w \in W}$, the family $\{X_{I_w}\}_{w \in W}$ (resp. the family $\{Y_{I_w}\}_{w \in W})$) is a basis of $\DcF$ as a left $S$-module. 

\medskip

We show now that for some hyperbolic formal group law $F_h$, the formal Demazure algebra can be identified with the classical Iwahori-Hecke algebra.

\medskip

Recall that the Iwahori-Hecke algebra $\mathcal{H}$ of the symmetric group $S_{n+1}$ is a $\Z[t,t^{-1}]$-algebra with generators $T_i$, $i\in \{1,\ldots,n\}$, subject to the following relations:
\begin{enumerate}[\indent(A)]
\item \label{item:quadratic} $(T_i+t)(T_i-t^{-1})=0$ or, equivalently, $T_i^2=(t^{-1}-t)T_i+1$,
\item \label{item:Tcommute} $T_iT_j=T_jT_i$ for $|i-j|>1$ and 
\item \label{item:Tbraid} $T_i T_j T_i=T_j T_i T_j$ for $|i-j|=1$.
\end{enumerate}
(The $T_i$'s appearing in the definition of the Iwahori-Hecke algebra \cite[Def.~7.1.1]{CG10} correspond to $tT_i$ in our notation, where $t=q^{-1/2}$.)

\medskip

Following \cite[Def.~6.3]{HMSZ} let $\mathrm{D}_F$ denote the $R$-subalgebra of $\DcF$ generated by the elements $X_i$, $i\in \{1,\ldots,n\}$, only. By \cite[Prop.~7.1]{HMSZ}, over $R=\mathbb{C}$, if $F=F_a$ (resp. $F=F_m$), then $\mathrm{D}_F$ is isomorphic to the completion of the nil-Hecke algebra (resp. the 0-Hecke algebra) of Kostant-Kumar.  The following observation provides another motivation for the study of formal (affine) Demazure algebras.

\medskip

Let us consider the FGL of example \ref{elliptic_ex} with invertible $\mu_1$. After normalization we may assume $\mu_1=1$. Then its formal inverse is $\tfrac{x}{x-1}$, and since $(1+\mu_2 x_i x_j) x_{i+j}= x_i+x_j-x_i x_j$, the coefficient $\kappa_{ij}$ of relation \eqref{item:braid} is simply $\mu_2$: 
\begin{equation}\label{eq:kappaij}
\kappa_{ij}  =\tfrac{1}{x_{i+j}x_j} - \tfrac{1}{x_{i+j}x_{-i}} -\tfrac{1}{x_{i}x_{j}} 
 = \tfrac{x_i+ x_j -x_i x_j -x_{i+j}}{x_ix_jx_{i+j}}
 = \tfrac{(1+\mu_2 x_i x_j)x_{i+j}-x_{i+j}}{x_i x_j x_{i+j}} = \mu_2
\end{equation}

\begin{prop}\label{Heckepres}
Let $F_h$ be a normalized (\ie $\mu_1=1$) hyperbolic formal group law over an integral domain $R$ containing $\Z[t,t^{-1}]$, and let $a,b \in R$. Then the following are equivalent
\begin{enumerate}
\item[(1)] The assignment $T_i\mapsto aX_i+b,\; i\in \{1,\ldots,n\}$, defines an isomorphism of $R$-algebras $\mathcal{H}\ot_{\Z[t,t^{-1}]}R \to \mathrm{D}_F$.
\item[(2)] We have $a=t+t^{-1}$ or $-t-t^{-1}$ and $b=-t$ or $t^{-1}$ respectively. Furthermore $\mu_2(t+t^{-1})^2 = -1$; in particular, the element $t+t^{-1}$ is invertible in $R$.
\end{enumerate}
\end{prop} 

\begin{proof}  
 Assume there is an isomorphism of $R$-algebras given by $T_i\mapsto aX_i+b$. Then relations \eqref{item:commute} and \eqref{item:Tcommute} are equivalent and relation \eqref{item:quadratic} implies that
\[
0=(aX_i+b)^2+(t-t^{-1})(aX_i+b)-1=[a^2+2ab+a(t-t^{-1})]X_i+b^2+b(t-t^{-1})-1.
\]
Therefore $b=-t$ or $t^{-1}$ and $a=t^{-1}-t-2b=t+t^{-1}$ or $-t-t^{-1}$ respectively, since $1$ and $X_i$ are $S$-linearly independent in $\mathrm{D}_F \subseteq \DcF$.

Relations \eqref{item:Tbraid} and \eqref{item:square} then imply 
\[
\begin{split}
0 & =(aX_i+b)(aX_j+b)(aX_i+b)-(aX_j+b)(aX_i+b)(aX_j+b) \\
 & =a^3(X_iX_jX_i-X_jX_iX_j)+(a^2b+ab^2)(X_i-X_j).
\end{split}
\]
Therefore, by relation \eqref{item:braid} and \eqref{eq:kappaij}, we have $a^3\mu_2-a^2b-ab^2=0$ which implies that $0=a^2\mu_2-ab-b^2=(t+t^{-1})^2\mu_2 +1$.

Conversely, by substituting the values of $a$ and $b$, it is easy to check that the assignment is well defined, essentially by the same computations. It is an isomorphism since $a=\pm(t+t^{-1})$ is invertible in $R$. 
\end{proof}

\begin{rem} 
The isomorphism of Proposition~\ref{Heckepres} provides a presentation of the Iwahori-Hecke algebra with $t+t^{-1}$ inverted in terms of the Demazure operators on the formal group algebra $\FGR{R}{\cl}{F_h}$.
\end{rem}

\begin{rem}
In general, the coefficients $\mu_1$ and $\mu_2$ of $F_h$ can be parametrized as $\mu_1=\epsilon_1+\epsilon_2$ and $\mu_2=-\epsilon_1\epsilon_2$ for some $\epsilon_1,\epsilon_2\in R$. In~\ref{Heckepres} it corresponds to $\epsilon_1=\tfrac{t}{t+t^{-1}}$ and $\epsilon_2=\tfrac{t^{-1}}{t+t^{-1}}$ (up to a sign) and in this case \cite[Thm.~4.1]{BuHo} implies that $F_h$ does not correspond to a topological complex oriented cohomology theory (\ie a theory obtained from complex cobordism $MU$ by tensoring over the Lazard ring). Observe that  such $F_h$ still corresponds to an algebraic oriented cohomology theory in the sense of Levine-Morel.
\end{rem}


\section{The algebraic restriction to the fixed locus on $G/B$} \label{sec:algmom}

In the present section we define the algebraic counterpart of the restriction to $T$-fixed locus of $G/B$. 

\medskip

Consider the $S$-linear dual $\SWd=\Hom_S(\SW,S)$ of the twisted formal group algebra. Since $\{\de_w\}_{w\in W}$ is a basis for both $\SW$ and $\QW$, $\SWd$ can be identified with the free $S$-submodule of $\QWd$ with basis $\{f_w\}_{w\in W}$ or, equivalently, with the subset $\{f\in \QWd\mid f(\SW)\subseteq S\}$. 

\medskip

Since $\de_\al=1-x_\al X_\al$ for each $\al\in \RS$, there is a natural inclusion of $S$-modules $\eta\colon \SW\hookrightarrow \DcF$.  The elements $\{X_{I_w}\}_{w\in W}$ (and, hence, $\{Y_{I_w}\}_{w\in   W}$) form a basis of $\DcF$ as a left $S$-module by \cite[Prop.~7.7]{CZZ}. Observe that the natural inclusion $\SW\hookrightarrow \QW$ factors through $\eta$. Tensoring $\eta$ by $Q$ we obtain an isomorphism $\eta_Q\colon \QW \stackrel{\simeq}\to Q \ot_S \DcF$, because both are free $Q$-modules and their bases $\{X_{I_w}\}_{w\in W}$ are mapped to each other. 

\medskip

\begin{defi} \label{defi:algres} Consider the $S$-linear dual $\DcFd=\Hom_{S}(\DcF, S)$. The induced map $\eta^\star\colon \DcFd \to \SWd$ (composition with $\eta$) will be called the {\em restriction to the fixed locus}.
\end{defi}

\begin{lem} \label{lem:DcFdinQW*} 
The map $\eta^\star$ is an injective ring homomorphism and its image in $\SWd \subseteq \QWd=Q\ot_S \SWd$ coincides with the subset
\[
\{f\in \SWd \mid f(\DcF)\subseteq S\}.
\]
Moreover, the basis of $\DcFd$ dual to $\{X_{I_w}\}_{w\in W}$ is $\{X_{I_w}^*\}_{w\in W}$ in $\QWd$.
\end{lem}
\begin{proof}
The coproduct $\QWcopr$ on $\QW$ restricts to a coproduct on $\DcF$ by \cite[Theorem~9.2]{CZZ} and to the coproduct on $\SW$ via $\eta$. 
Hence, the map $\eta^\star$ is a ring homomorphism. 

There is a commutative diagram
\[
\xymatrix{
\DcFd \ar[d] \ar[r]^{\eta^\star} & \SWd\ar[d] \\
Q\ot_S \DcFd\ar[r]^-{\eta^\star_Q}_-{\simeq} & Q\ot_S \SWd
}
\]
where the vertical maps are injective by freeness of the modules and because $S$ injects into $Q$. The description for the image then follows from the fact that $\{X_{I_w}\}_{w\in W}$ is a basis for both $\DcF$ and $\QW$.

The last part of the lemma follows immediately.
\end{proof}

\medskip

By Lemma~\ref{lem:DcFdinQW*}, $\sigma\in \DcFd\subseteq \QWd$ means that $\sigma(\DcF)\subseteq S$. For any $X\in \DcF$ we have $(X\act \sigma)(\DcF)=\sigma(\DcF X)\subseteq S$, so $X\act \sigma\in \DcFd$.  Hence,  the `$\act$'-action of $\QW$ on $\QWd$ induces a `$\act$'-action of $\DcF$ on $\DcFd$.

\medskip

For each $v\in W$,  we define
\[
\tf_v: = x_\Pi\act f_v = v(x_\Pi) f_v \in \QWd, \quad \text{\ie}\; \tf_v(\sum_{w\in W} q_w\de_w)=v(x_\Pi) q_v.
\]
 
\begin{lem}\label{lem:point}
We have $\tf_v\in \DcFd$ for any $v\in W$.
\end{lem}

\begin{proof} We know that $x_\Pi=w_0(x_{w_0})$, and by Lemma \ref{lem:Sigmaw}.\eqref{item:cw0Sigma}, $\frac{x_{w_0}}{v(x_{w_0})}$ is invertible in $S$ for any $v\in W$, so it suffices to show that $x_\Pi f_v\in \DcFd$. If $v=w_0$, by Lemma~\ref{lem:Xdelta}, we have
\[
X_{I_{w_0}}=\sum_{w\le w_0}a^X_{w_0,w}\de_w, \text{ where } a^X_{w_0,w_0}=(-1)^N\tfrac{1}{x_{w_0}},\;\text{ so}
\] 
\[
(x_\Pi f_{w_0})(X_{I_u})=(x_\Pi  f_{w_0})(\sum_{w\le u}a^X_{u,w}\de_w)=(x_\Pi a^X_{w_0,w_0})\Kr_{u,w_0}=(-1)^N\tfrac{x_\Pi}{x_{w_0}}\Kr_{u,w_0}\in S.
\] 
By Lemma~\ref{lem:DcFdinQW*}, we have $x_\Pi f_{w_0}\in \DcFd$. For an arbitrary $v\in W$, by Lemma~\ref{lem:bulletactprop}, we obtain
\[
x_\Pi f_v=x_\Pi f_{w_0w_0^{-1}v}=v^{-1}w_0(x_\Pi f_{w_0})=v^{-1}w_0(x_\Pi f_{w_0})\in \DcFd. \qedhere
\]
\end{proof}

\begin{cor}\label{cor:xPiDinSW}
For any $z\in \DcF$, we have $x_\Pi z\in \SW$ and $z x_\Pi\in \SW$.
\end{cor}

\begin{proof} 
It suffices to show that for any sequence $I_v$, $x_\Pi X_{I_v}$ and $X_{I_v}x_\Pi$ belong to $S_W$. Indeed,  
\[
x_\Pi X_{I_v}=x_\Pi\sum_{w\le v}a^X_{v,w}\de_w=\sum_{w\le v}(x_\Pi a^X_{v,w})\de_w=\sum_{w\le v}(x_\Pi f_{w})(X_{I_v})\de_w\in S_W,
\]
and 
\[
X_{I_v}x_\Pi=\sum_{w\le v}a^X_{v,w}\de_wx_\Pi=\sum_{w\le v}a^X_{v,w}w(x_\Pi)\de_w=\sum_{w\le v}(w(x_\Pi)f_w)(X_{I_v})\de_w\in \SW.\qedhere
\]
\end{proof}

Let $\zeta\colon \DcF \to \SW$ be the multiplication on the \emph{right} by $x_{\Pi}$ (it does indeed land in $\SW$ by Corollary \ref{cor:xPiDinSW}). The dual map $\zeta^\star\colon \SWd \to \DcF^\star$ is the `$\act$'-action by $x_\Pi$, and $\zeta^\star(f_v)=\tf_v$. 

\begin{rem} 
In $T$-equivariant cohomology, the map $\zeta^\star$ corresponds to the push-forward from the $T$-fixed point set of $G/B$ to $G/B$ itself, see \cite[Lemma 8.5]{CZZ2}.
In the topological context, for singular cohomology, it coincides with the map $i_*$ discussed in \cite[p.8]{AB84}. 
\end{rem}

\begin{lem} \label{lem:maxsubmod} 
The unique maximal left $\DcF$-module (by the $\act$-action) that is contained in $\SWd$ is $\DcFd$.
\end{lem}
\begin{proof}
Let $f$ be any element in a given $\DcF$-module $M$ contained in $\SWd$. Then $X_I \act f \in M \subseteq \SWd$ for any sequence $I$, and $(X_I \act f)(\de_e) = f(X_I) \in S$. Since $X_I$'s generate $\DcF$ as an $S$-module, we have $f(\DcF) \subseteq S$, and therefore $f \in \DcF^\star$ by Lemma~\ref{lem:DcFdinQW*}. 
\end{proof}

Define the $S$-module
\[
\mZ=\{f \in \SWd \mid B_i(f) \in \SWd \text{ for any simple root $\al_i$ }\}.
\]
Since for an element $f = \sum_{w\in W}q_w f_w$, $q_w\in S$ we have
\[
B_{i}(f)=X_i \act f = \sum_{w \in W} \tfrac{q_w-q_{ws_i}}{x_{w(\al_i)}} f_w=\sum_{w\in W}\tfrac{q_w-q_{s_{w(\al_i)}w}}{x_{w(\al_i)}}f_w,
\]
this can be rewritten as
\[
\mZ=\{\sum_{w\in W}q_w f_w\in \SWd \mid  \tfrac{q_w-q_{s_\al w}}{x_\al}\in S \text{ for any root }\al\text{ and any }w\in W \}.
\]

The following theorem provides another characterization of $\DcFd$

\begin{theo} \label{theo:D=Z}
We have $\DcFd \subseteq \mZ$, and under the conditions of Lemma \ref{lem:div}, we have $\DcFd=\mZ$.
\end{theo}

\begin{proof} 
Since $\DcFd\subseteq\SWd$ is a sub-$\DcF$-module, we have $\DcFd \subseteq \mZ$. By Lemma~\ref{lem:maxsubmod}, $\DcFd$ is the unique maximal $\DcF$-module contained in $\SWd$, so  we only need to prove that $\mZ$ is a $\DcF$-submodule. 

It suffices to show that for any $f \in \mZ$ and for any simple root $\alpha_i$, the element $X_i \act f$ is still in $\mZ$, or in other words, that for any two simple roots $\al_i$ and $\al_j$, we still have $X_iX_j \act f \in \SWd$. 
If $\alpha_i=\al_j$, it follows from $X_i^2= \kp_i X_i$. 

If $s_j(\al_i)=\al_i$, then $s_is_j=s_js_i$. Let $f=\sum_{w\in W}q_wf_w$, then $ X_i\act f=\sum_{w\in W}\frac{q_w-q_{ws_i}}{x_{w(\al_i)}}f_w$. Set $p_w=\frac{q_w-q_{ws_i}}{x_{w(\al_i)}}$, then 
\[
(X_jX_i)\act f=\sum_{w\in W}\tfrac{p_w-p_{ws_j}}{x_{w(\al_j)}}f_w=\sum_{w\in W}\tfrac{q_w-q_{ws_i}-q_{ws_j}+q_{ws_js_i}}{x_{w(\al_i)}x_{w(\al_j)}}f_w.
\]
Rearranging the numerator, we see that it is divisible by both  $x_{w(\al_i)}$ and $x_{w(\al_j)}$, so it is divisible by $x_{w(\al_i)}x_{w(\al_j)}$ by Lemma~\ref{lem:div}.

Suppose $s_j(\alpha_i) \neq \al_i$, then $s_j(\al_i) \neq  \al_j$.  Since $X_i\act f=\sum_w p_w f_w$ with $p_w\in S$ as above, we need to prove that the coefficient of $f_w$ in $X_jX_i\act f$ is in $S$, for any $w$. This coefficient is
\[
\tfrac{p_w - p_{ws_j}}{x_{w(\al_j)}} = \tfrac{(q_w-q_{ws_i})x_{ws_j(\al_i)}-(q_{ws_j} -q_{ws_j s_i})x_{w(\al_i)}}{x_{w(\al_i)}x_{w(\al_j)}x_{ws_j(\al_i)}}.
\]
Since the numerator is already divisible by $x_{w(\al_i)}$ and by $x_{ws_j(\al_i)}$ by assumption, it suffices, by Lemma~\ref{lem:div}, to show that it is divisible by $x_{w(\al_j)}$. Setting $\gamma=w(\al_j)$ and $\nu=w(\alpha_i)$, it becomes  
$(q_w-q_{s_\nu w})x_{s_\gamma(\nu)}-(q_{s_\gamma w} -q_{s_\gamma s_\nu w})x_{\nu}$. Using that $x_{s_\gamma(\nu)}=F(x_\nu,x_{-\langle \nu, \gamma^\vee\rangle \gamma})\equiv x_{\nu} \mod x_{\gamma}$, the numerator is congruent to  (cf. the proof of \cite[Lem.~5.7]{HMSZ})
\[
((q_w-q_{s_\gamma w})-(q_{s_\nu w}-q_{s_\gamma s_\nu w}))x_{\nu}
\]
which is  $0 \mod x_{\gamma}$, by assumption.
\end{proof}

\begin{rem}
The geometric translation of this theorem (\cite[Theorem 9.2]{CZZ2})
generalizes the classical result \cite[Proposition 6.5.(i)]{Br97}. 
\end{rem}

\begin{rem}
In Theorem \ref{theo:D=Z}, it is not possible to remove entirely the assumptions on the root system and the base ring, as the following example shows. Take a root datum of type $G_2$, and a ring $R$ in which $3=0$, with the additive formal group law $F$ over $R$. Then, $S$ is $\RS$-regular, and if $(\al_1,\al_2)$ is a basis of simple roots, with $\beta=2 \al_2+3 \al_1$ being the longest root, we have $x_\beta=2 x_{\al_2}=-x_{\al_2}$. It is not difficult to check that the element $f=(\prod_{\al \in \RS^+,\ \al \neq \beta} x_\alpha) f_e$ is in $\mZ$, but 
\[
f( X_{I_{w_0}})\overset{\text{Lem}. \ref{lem:Xdelta}}=(\prod_{\al\in \RS^+, \al\neq \beta}x_\al)\cdot \tfrac{1}{x_{w_0}}=\tfrac{1}{x_\beta}\not \in S,
\]
so $f\not\in \DcFd$. Therefore,  $\mZ\supsetneq\DcFd$. Indeed,  $\mZ$ is not even a $\DcF$-module.
\end{rem}

Recall from \eqref{eq:XYdualbasis} that $X_{I_w}^*=\sum_{v\geq w}b^X_{v,w}f_v$ and $Y_{I_w}^*=\sum_{v\ge w}b^Y_{v,w}f_v$.
\begin{cor} \label{cor:bXdiv}
For any  $v,w\in W$ and root $\al$, we have $x_\al\mid (b^X_{v,w}-b^X_{s_\al v,w})$ and $x_\al\mid (b^Y_{v,w}-b^Y_{s_\al v, w})$.
\end{cor}

\begin{rem}
It is not difficult to see that  Corollary \ref{cor:xPiDinSW}  and Corollary \ref{cor:bXdiv} provide a characterization of elements of $\DcF$ inside $\QW$. This characterization coincides with the residue description of $\DcF$ in \cite[\S4]{ZZ}, which generalizes Ginzburg--Kapranov--Vasserot's construction of certain Hecke algebras in \cite{GKV}.
\end{rem}
\medskip 
For any $\Xi\subseteq \Pi$ and $w\in W$, define 
\[
\Xa{\Xi}{I_w}=\sum_{v\in W_\Xi}\de_v\tfrac{b^X_{v^{-1},w}}{x_\Xi} \quad \text{and}\quad \Ya{\Xi}{I_w}=\sum_{v\in W_\Xi} \de_v\tfrac{b^Y_{v^{-1}, w}}{x_\Xi}.
\]
By Lemma \ref{lem:Xdelta}, $b^X_{v,e}=1$, so 
\[
\Xa{\Xi}{\emptyset}=\sum_{v\in W_\Xi}\de_v\tfrac{1}{x_\Xi}=Y_\Xi.
\]
Note that $Y_\Xi$ does not depend on the choice of reduced sequences $\{I_w\}_{w\in W}$, but $\Xa{\Xi}{I_w}$ and $\Ya{\Xi}{I_w}$ do, since $b^X_{w,v}$ and $b^Y_{w,v}$ do for $w$ such that $\ell(w)\ge 3$. Moreover, we have 
\begin{equation} \label{eq:XwXstar}
\Xa{\Pi}{I_w} \act \tf_e = X_{I_w}^* \quad \text{and}\quad \Ya{\Pi}{I_w}\act \tf_e = Y_{I_w}^*
\end{equation}
by a straightforward computation.

\begin{lem}\label{lem:ZwinDcF} 
For any $\Xi\subseteq \Pi$ and $w\in W$, we have $\Xa{\Xi}{I_w}\in \DcF$ and $\Ya{\Xi}{I_w}\in \DcF$. 
\end{lem}

\begin{proof}
The ring $\QW$ is functorial in the root datum (\ie along morphisms of lattices that send roots to roots) and in the formal group law. This functoriality sends elements $X_\al$ (or $Y_\al$) to themselves, so it restricts to a functoriality of the subring $\DcF$. It also sends the elements $\Xa{\Xi}{I_w}$ (or $\Ya{\Xi}{I_w}$) to themselves. We can therefore assume that the root datum is adjoint, and that the formal group law is the universal one over the Lazard ring, in which all integers are regular, since it is a polynomial ring over $\Z$.

Consider the involution $\iota$ on $\QW$ given by $q\de_w\mapsto (-1)^{\ell(w)}w^{-1}(q)\de_{w^{-1}}$. It satisfies $\iota(zz')=\iota(z')\iota(z)$. Since $\iota(X_\al)=Y_{-\al}$, it restricts to an involution on $\DcF$. 

To show that $\Xa{\Xi}{I_w}\in \DcF$, it suffices to show that $\iota(\Xa{\Xi}{I_w})\in \DcF$. We have 
\[
\iota(\Xa{\Xi}{I_w})=\sum_{v\in W_\Xi}(-1)^{\ell(v)}\tfrac{b^X_{v^{-1},w}}{x_\Xi}\de_{v^{-1}}=\tfrac{1}{x_\Xi}\sum_{v\in W_\Xi}(-1)^{\ell(v)}b^X_{v,w}\de_v.
\]
Since the root datum is adjoint, we have $\DcF=\{f\in \QW\mid f\cdot S\subseteq S\}$ by \cite[Remark~7.8]{CZZ}, so it suffices to show that $\iota(\Xa{\Xi}{I_w})\cdot x\in S$ for any $x\in S$. We have 
\[
\iota(\Xa{\Xi}{I_w})\cdot x=\tfrac{1}{x_\Xi}\sum_{v\in W_\Xi}(-1)^{\ell(v)}b^X_{v,w}v(x).
\]
By Lemma~\ref{lem:div}, it is enough to show that $\sum_{v\in W}(-1)^{\ell(v)}b^X_{v,w}v(x)$ is divisible by $x_\al$ for any root $\al\in \RS_\Xi^-$. Let $^\al W_\Xi=\{v\in W_\Xi|\ell(s_\al v)>\ell(v)\}$. Then $(-1)^{\ell(s_\al v)} = -(-1)^{\ell(v)}$ and $W_\Xi= {^\al W}_\Xi \sqcup s_\al {^\al W}_\Xi$. So
\[
\begin{split}
& \sum_{v\in W_\Xi}(-1)^{\ell(v)}b^X_{v,w}v(x)=\sum_{v\in {}^\al W_\Xi}(-1)^{\ell(v)}(b^X_{v,w}v(x)-b^X_{s_\al v,w} s_\al v(x))\\
& =\sum_{v\in {}^\al W_\Xi}(-1)^{\ell(v)}(b^X_{v,w}v(x)-b^X_{v,w}s_\al v(x)+b^X_{v,w}s_{\al} v(x)-b^X_{s_\al v,w}s_\al v(x)) \\
& =\sum_{v\in {}^\al W_\Xi}(-1)^{\ell(v)}\bigl(b^X_{v,w}x_\al\Dem_\al\big(v(x)\big)+(b^X_{v,w}-b^X_{s_\al v,w})s_\al v(x)\bigl)
\end{split}
\]
which is divisible by $x_\al$ by Corollary \ref{cor:bXdiv}. Therefore $\Xa{\Xi}{I_w}\in \DcF$. The proof that $\Ya{\Xi}{I_w}\in \DcF$ is similar.
\end{proof}

\begin{theo}
$\QWd$ is a free $\QW$-module of rank 1 generated by $f_w$ for any $w\in W$, and $\DcFd$ is a free left $\DcF$-module of rank 1 generated by $\tf_w$ for any $w\in W$. 
\end{theo}

\begin{proof}
Since $\de_v \act f_w=f_{wv^{-1}}$, we have $Q_W \act f_w = \QWd$. Moreover, if $z=\sum_{v\in W}q_v\de_v$ such that $z\act f_w=0$, then $\sum_{v\in W}q_vf_{wv^{-1}}=0$, so $q_v=0$ for all $v\in W$, \ie $z=0$; the first part is proven. 

To prove the second part, note that by Lemma \ref{lem:point} $\tf_e\in \DcFd$  for any $w$. Moreover,  $\{\tf_e\}$ is $\QW$-linearly independent by the first part of the proof, hence it is $\DcF$-linearly independent. On the other hand, $\DcF\act \tf_e=\DcFd$ by Lemma~\ref{lem:ZwinDcF} and \eqref{eq:XwXstar}, so $\tf_e$ generates $\DcFd$ as a left $\DcF$-module. Since $\tf_w=\tfrac{x_\Pi}{w^{-1}(x_\Pi)}\de_{w^{-1}}\act \tf_e$, and $\tfrac{x_\Pi}{w^{-1}(x_\Pi)} \in S$ by Lemma \ref{lem:Sigmaw}.\eqref{item:cw0Sigma}, the same is true for $\tf_w$.
\end{proof}

 
\section{The algebraic restriction to the fixed locus on $G/P$} \label{sec:algresP}

We now extend the results of the previous section to the relative case of $W/W_\Xi$.

\medskip

For any  $\Xi\subseteq \Pi$, let $\SWP{\Xi}$ be the free $S$-module with basis $(\de_{\bar{w}})_{\bar{w}\in W/W_\Xi}$ (it is not necessarily a ring). Let $\QWP{\Xi}=Q \otimes_S \SWP{\Xi}$ be its $Q$-localization. There is a left $S$-linear coproduct on $\SWP{\Xi}$, defined on basis elements by the formula $\de_{\bar{w}} \mapsto \de_{\bar{w}} \otimes \de_{\bar{w}}$; it extends by the same formula to a $Q$-linear coproduct on $\QWP{\Xi}$. The induced products on the $S$-dual $\SWPd{\Xi}$ and the $Q$-dual $\QWPd{\Xi}$ are given by the formula $f_{\bar{v}} f_{\bar{w}} = \Kr_{\bar{v},\bar{w}} f_{\bar{v}}$.

\medskip

If $\Xi' \subseteq \Xi$ and $\bar{w} \in W/W_{\Xi'}$, let $\hat{w}$ its class in $W/W_{\Xi}$. We consider the projection and the sum over orbit maps 
\[
\begin{array}[t]{rccc}
\p{\Xi/\Xi'} : & \SWP{\Xi'} & \to & \SWP{\Xi} \\
 & \de_{\bar{w}} & \mapsto & \de_{\hat{w}}
\end{array}
\quad
\text{and}
\quad
\begin{array}[t]{rccc}
\dd{\Xi/\Xi'} : & \SWP{\Xi} & \to & \SWP{\Xi'} \\
 & \de_{\hat{w}} & \mapsto & \sum\limits_{\substack{\bar{v} \in W/W_{\Xi'} \\ \hat{v}=\hat{w}}} \de_{\bar{v}}.
\end{array} \vspace{-3ex}
\] 
with $S$-dual maps
\[
\arraycolsep=.4ex
\begin{array}[t]{rccc}
\pd{\Xi/\Xi'} : & \SWPd{\Xi} & \to & \SWPd{\Xi'} \\
 & f_{\hat{w}} & \mapsto & \sum\limits_{\substack{\bar{v} \in W/W_{\Xi'} \\ \hat{v}=\hat{w}}} f_{\bar{v}}
\end{array}
\quad
\text{and}
\quad
\begin{array}[t]{rccc}
\ddd{\Xi/\Xi'} : & \SWPd{\Xi'} & \to & \SWPd{\Xi} \\
 & f_{\bar{w}} & \mapsto & f_{\hat{w}}.
\end{array}
\] 

We use the same notation for maps between the corresponding $Q$-localized module $\QWP{\Xi}$ and $\QWP{\Xi'}$, and we write $\pqd{\Xi/\Xi'}$ and $\ddqd{\Xi/\Xi'}$ for their $Q$-dual maps. As usual, when $\Xi'=\emptyset$, we omit it, as in $\p{\Xi}: \SW \to \SWP{\Xi}$. Note that the maps $\p{\Xi/\Xi'}$ preserve the coproduct (the maps $\dd{\Xi/\Xi'}$ don't), and thus the dual maps $\pd{\Xi/\Xi'}$ and $\pqd{\Xi/\Xi'}$ are ring maps. We set $\DcFP{\Xi}:=\p{\Xi}(\DcF)\subseteq \QWP{\Xi}$. 

\medskip

The coproduct on $\QWP{\Xi}$ therefore restricts to a coproduct on $\DcFP{\Xi}$.
We then have the following commutative diagram of $S$-modules	which defines the map $\eta_\Xi$	
\begin{equation} \label{eq:petadiag}
\begin{gathered}
\xymatrix{
\SW\ar@{^(->}[r]^\eta  \ar@{->>}[d]_{\p{\Xi}} & \DcF \ar@{^(->}[r] \ar@{->>}[d]_{\p{\Xi}} & \QW\ar@{->>}[d]_{\p{\Xi}} \\
S_{W/W_\Xi} \ar@{^(->}[r]^-{\eta_\Xi} & \DcFP{\Xi} \ar@{^(->}[r] & Q_{W/W_\Xi}.
}
\end{gathered}
\end{equation}

\begin{lem}
The map $\p{\Xi/\Xi'}\colon \QWP{\Xi'} \to \QWP{\Xi}$ restricts to $\DcFP{\Xi'} \to \DcFP{\Xi}$. 
\end{lem}

\begin{proof}
It follows by diagram chase from Diagram \eqref{eq:petadiag} applied first to $\Xi$ and then to $\Xi'$, using the surjectivity of $\p{\Xi'}\colon \DcF \to \DcFP{\Xi'}$. 
\end{proof}

\begin{lem}\label{lem:pXi}
We have $\p{\Xi}(zX_{\al})=0$ for any $\al\in \RS_\Xi$ and $z\in \QW$.
\end{lem}

\begin{proof}
Since $\p{\Xi}$ is a map of $Q$-modules, it suffices to consider $z=\de_w$, in which case $\de_w X_\al=\frac{1}{w(x_\al)}\de_w-\frac{1}{w(x_\al)}\de_{ws_\al}$, so $p(\de_wX_\al)=\frac{1}{w(x_\al)}(\de_{\bar{w}}-\de_{\bar{w}})=0$.
\end{proof}

For any $w \in W$, let $X^\Xi_{I_w}$ be the element $\p{\Xi}(X_{I_w}) \in \DcFP{\Xi}$.

\begin{lem}\label{lem:DcFPbasis}
\begin{enumerate}
\item \label{item:zeroX} Let $\{I_w\}_{w\in W}$ be a family of $\Xi$-compatible reduced sequences. If $w \notin W^\Xi$, then $X^\Xi_{I_w}=0$. 
\item \label{item:basisX} Let $\{I_w\}_{w\in W^\Xi}$ be a family of reduced sequences of minimal length. Then the family $\{X^\Xi_{I_{w}}\}_{w\in W^\Xi}$ forms a $S$-basis of $\DcFP{\Xi}$ and, therefore, forms also a $Q$-basis of $Q_{W/W_\Xi}$. 
\end{enumerate}
\end{lem}

\begin{proof}
\eqref{item:zeroX} If $w\not \in W^\Xi$, then $w=uv$ with $u\in W^\Xi$ and $e\neq v\in W_\Xi$. By Lemma~\ref{lem:pXi}, we have $\p{\Xi}(X_{I_w})=\p{\Xi}(X_{I_u}X_{I_v})=0$. 

\medskip

\eqref{item:basisX} Let us complete $\{I_w\}_{w \in W^\Xi}$ to a $\Xi$-compatible choice of reduced sequences $\{I_w\}_{w \in W}$ by choosing reduced decompositions for elements in $W_{\Xi}$.  Since $\{X_{I_w}\}_{w\in W}$ is a basis of $\DcF$, its image $\DcFP{\Xi}$ in $\QWP{\Xi}$ is spanned by $\{X^\Xi_{I_w}\}_{w\in W^\Xi}$ by part (a). Writing $X_{I_w}=\sum_{v\le w}a^X_{w,v}\de_v$ yields $X^\Xi_{I_w}=\sum_{v\le w}a^X_{w,v}\de_{\bar{v}}$.  Since $w\in W^\Xi$ is of minimal length in $wW_\Xi$,  the coefficient of $\de_{\bar{w}}$ in $X^\Xi_{I_w}$ is $a^X_{w,w}=(-1)^{\ell(w)}\frac{1}{x_w}$, invertible in $Q$, so the matrix expressing the $\{X^\Xi_{I_w}\}_{w\in W^\Xi}$ on the basis $\{\de_{\bar{w}}\}_{w\in W^\Xi}$ is upper triangular with invertible (in $Q$) determinant, hence $\{X^\Xi_{I_w}\}_{w\in W^\Xi}$ is $Q$-linearly independent in $Q_{W/W_\Xi}$ and therefore $S$-linearly independent in $\DcFP{\Xi}$.
\end{proof}

Observe in particular that $\DcFP{\Pi} \simeq S$ carried by $X_{\emptyset}^\Pi=\de_{\bar{e}}$.

\begin{defi} \label{defi:algresP}
The dual map $\eta_\Xi^\star: \DcFPd{\Xi} \to \SWPd{\Xi}$ is called the \emph{algebraic restriction to the fixed locus}.
\end{defi}

As in Lemma \ref{lem:DcFdinQW*}, and by the similar proof, we obtain:
\begin{lem} \label{lem:algresPinj}
The map $\eta_\Xi^\star$ is an injective ring homomorphism and its image in $\SWPd{\Xi} \subseteq \QWPd{\Xi}$ coincides with the subset
\[
\{f\in \SWPd{\Xi} \mid f(\DcFP{\Xi})\subseteq S\}.
\]
Moreover, the basis of $\DcFPd{\Xi}$ dual to $\{X^\Xi_{I_w}\}_{w\in W^\Xi}$ maps to $\big\{(X^\Xi_{I_w})^*\big\}_{w\in W^\Xi}$ in $\QWPd{\Xi}$.
\end{lem}

So far, the situation is summarized in the diagram of $S$-linear ring maps
\begin{equation}
\begin{gathered}
\xymatrix{ 
\DcFd \ar@{^(->}[r]^{\eta^\star}  & \SWd \\ 
\DcFPd{\Xi} \ar@{^(->}[r]^{\eta_\Xi^\star} \ar@{^(->}[u]^{\pd{\Xi}} & S_{W/W_\Xi}^\star \ar@{^(->}[u]_{\pd{\Xi}} \\ 
}
\end{gathered}
\end{equation}
in which both columns become injections $\QWPd{\Xi} \hookrightarrow \QWd$ after $Q$-localization. The geometric translation of this diagram is in the proof of Corollary~8.7 in \cite{CZZ2}.

\begin{lem} \label{lem:DdWXibasis}
For any $\Xi$-compatible choice of reduced sequences $\{I_w\}_{w \in W}$, the $W_\Xi$-invariant subring $(\DcFd)^{W_\Xi}$ is a free $S$-module with basis $\{X_{I_w}^\star\}_{w \in W^\Xi}$.
\end{lem}

\begin{proof}
It follows from Corollary \ref{cor:invbasisq} since $(\DcFd)^{W_\Xi}=(\QWd)^{W_\Xi}\cap \DcFd$.
\end{proof}

\begin{lem} \label{lem:DFinv}
The injective maps $\pd{\Xi}\colon\SWPd{\Xi} \to \SWd$, $\pqd{\Xi}\colon \QWPd{\Xi} \to \QWd$ and $\pd{\Xi}\colon\DcFPd{\Xi} \to \DcFd$ have images $(\SWd)^{W_\Xi}$, $(\QWd)^{W_\Xi}$ and $(\DcFd)^{W_\Xi}$, respectively.
\end{lem}

\begin{proof}
For any $w\in W$, we have $\pqd{\Xi}(f_{\bar{w}})=f^\Xi_w$. Thus $\pqd{\Xi}(\QWPd{\Xi})=(\QWd)^{W_\Xi}$ by Lemma~\ref{lem:QWinvariant}. Similarly, $\pd{\Xi}(\SWPd{\Xi})=(\SWd)^{W_\Xi}$. Finally, take a $\Xi$-compatible choice of reduced sequences $\{I_w\}_{w\in W}$, dualizing the fact that $\p{\Xi}(X_{I_w})=X^\Xi_{I_w}$, which, by Lemma \ref{lem:DcFPbasis}, is $0$ if $w\not\in W^\Xi$ and a basis element otherwise, we obtain that $\pd{\Xi}((X^\Xi_{I_w})^\star)=X^\star_{I_w}$ if $w \in W^\Xi$, and thus the conclusion for $\DcFPd{\Xi}$ by Lemma \ref{lem:DdWXibasis}. 
\end{proof}
\begin{rem}Note that if $\{I_w\}_{w\in W}$ is not $\Xi$-compatible, then we may not have $\pd{\Xi}((X^\Xi_{I_w})^\star)=X^\star_{I_w}$ for all $w \in W^\Xi$.
\end{rem}

Through the resulting isomorphism $\DcFPd{\Xi} \simeq (\DcFd)^{W_\Xi}$, we obtain
\[
\begin{split}
\DcFPd{\Xi} & =\{f\in \SWPd{\Xi}\mid f(\DcFP{\Xi})\subseteq S\} \\
\simeq (\DcFd)^{W_\Xi} & =\{f\in (\SWd)^{W_\Xi}\mid f(\DcF)\subseteq S\} \\ 
 & =\{f\in \SWd\mid f(\DcF)\subseteq S \text{ and } f(K_\Xi)=0 \} 
\end{split}
\]
where $K_\Xi$ is the kernel of $\p{\Xi}$, \ie the sub-$S$-module of $\DcF$ generated by $(X_{I_w})_{w \notin W^\Xi}$ for a $\Xi$-compatible choice of reduced sequences $\{I_w\}_{w\in W}$.
\medskip

Since $(\DcFd)^{W_\Xi}=\DcFd \cap (\SWd)^{W_\Xi}$, an element of $\SWPd{\Xi}$ is in $\DcFPd{\Xi}$ if and only if its image by $\pd{\Xi}$ is in $\DcFd$. Since $B_\al(f)=0$ when $f \in (\SWd)^{W_\Xi}$ and $\alpha \in W_\Xi$, Theorem \ref{theo:D=Z} then gives:

\begin{theo}
Under the conditions of Lemma \ref{lem:div}, an element $f \in \SWPd{\Xi}$ is in $\DcFPd{\Xi}$ if and only if $B_\al \circ \pd{\Xi}(f) \in \SWd$ for any $\al \notin \RS_\Xi$. In other words, $f = \sum_{\bar{w}} q_{\bar{w}} f_{\bar{w}}$ is in $\DcFPd{\Xi}$ if and only if $x_{w(\al)}$ divides $q_{\bar{w}} - q_{\overline{s_{w(\al)w}}}$ for any $\bar{w} \in W/W_\Xi$ and any $\al \notin \RS_\Xi$. 
\end{theo}


\section{The push-pull operators on $\DcFd$}\label{sec:pushpullDdual}

In this section we restrict the push-pull operators onto the dual of the formal affine Demazure algebra $\DcFd$, and define a non-degenerate pairing on it.

\medskip

By Lemma \ref{lem:ZwinDcF}, we have $Y_\Xi\in \DcF$, so 
\begin{cor}\label{cor:opAinv}
The operator $Y_\Xi$ (resp. $A_\Xi$) restricted to $S$ (resp. to $\DcFd$) defines an operator on $S$ (resp. on $\DcFd$).  Moreover, we have
\[
C_\Xi(S)\subseteq S^{W_\Xi}\quad\text{ and }\quad A_\Xi(\DcFd)\subseteq (\DcFd)^{W_\Xi}.
\]
\end{cor}

\begin{proof} 
Here $Y_\Xi$ acts on $S\subseteq Q$ via \eqref{eq:leftactQW}. Since $Y_\Xi\in \DcF\subseteq \{z\in \QW\mid z\cdot S\subseteq S\}$ by \cite[Remark~7.8]{CZZ} and $Y_\Xi\cdot Q \subseteq (Q)^{W_\Xi}$, the result follows. 

As for $A_\Xi$, by Lemma~\ref{lem:DcFdinQW*} any $f\in \DcFd$ has the property that $f(\DcF)\subseteq S$. Therefore, $(A_\Xi(f))(\DcF)=(Y_\Xi\act f)(\DcF)=f(\DcF Y_\Xi)\subseteq S$, so $A_\Xi(f)\in \DcFd$. The result then follows by Lemma~\ref{lem:projformulaA}.
\end{proof}

\begin{cor} \label{cor:opAinv2}
Suppose that the root datum has no irreducible component of type $C_n^{sc}$ or that $2$ is invertible in $R$. Then if $|W_{\Xi'}|$ is regular in $R$, for any $\Xi'\subseteq \Xi\subseteq \Pi$, we have 
\[
C_{\Xi/\Xi'}(S^{W_{\Xi'}})\subseteq S^{W_\Xi}.
\]
\end{cor}

\begin{proof} 
Let $x\in S^{W_{\Xi'}}$, then $|W_{\Xi'}|\cdot x=\sum_{w\in W_{\Xi'}}w(x)$. So we have 
\[
|W_{{\Xi'}}|\cdot C_{\Xi/\Xi'}(x)=C_{\Xi/\Xi'}(|W_{\Xi'}|\cdot x)=\sum_{u\in W_{\Xi/\Xi'}}u(\tfrac{|W_{\Xi'}|\cdot x}{x_{\Xi/\Xi'}})
\]
$$=\sum_{u\in W_{\Xi/\Xi'}}\sum_{v\in W_{\Xi'}}uv(\tfrac{x}{x_{\Xi/\Xi'}})=\sum_{w\in W_\Xi}w(\tfrac{xx_{\Xi'}}{x_\Xi})\in S^{W_\Xi}.$$
Thus $|W_{\Xi'}|\cdot C_{\Xi/\Xi'}(x)\in S$, which implies that $C_{\Xi/\Xi'}(x)\in S$ by \cite[Lemma 3.5]{CZZ}. Besides, it is fixed by $W_{\Xi}$ by Lemma~\ref{lem:Cinvtoinv}. 
\end{proof}

\begin{cor}
If $|W|$ is invertible in $R$, then $C_{\Xi/\Xi'}( S^{W_{\Xi'}})=S^{W_\Xi}$.
\end{cor}
\begin{proof}From the proof of Corollary \ref{cor:opAinv2} we know that for any $x\in S^{W_{\Xi'}}$, $|W_{\Xi}| x=C_{\Xi}\cdot (xx_\Xi)$, so $C_\Xi(S)=S^{W_\Xi}$. The conclusion then follows from the identity $C_{\Xi/\Xi'}\circ C_{\Xi'}=C_\Xi$ of Lemma \ref{lem:pullpushcomp}.
\end{proof}

\begin{theo} \label{theo:bilform}
For any $v,w\in W$, we have
\[
A_\Pi(Y_{I_v}^\star A_{I_w^{\rev}}(\tf_e))=\Kr_{w,v}\unit=A_\Pi(X_{I_v}^\star B_{I_w^{\rev}}(\tf_e)).
\]
Consequently, the pairing
\[
A_\Pi\colon \DcFd\times \DcFd \to (\DcFd)^{W}\cong S, \quad (\sigma,\sigma')\mapsto A_\Pi(\sigma\sigma')
\]
is non-degenerate and 
satisfies that $(A_{I_w^{\rev}}(\tf_e))_{w\in W}$ is dual to the basis $(Y_{I_v}^\star)_{v\in W}$, while  $(B_{I_w^{\rev}}(\tf_e))_{w\in W}$ is dual to the basis $(X_{I_v}^\star)_{v\in W}$.
\end{theo}

\begin{proof}
We prove the first identity. The second identity is obtained similarly. 

Let $Y_{I_w^{\rev}}=\sum_{v\in W}a'_{w,v}\de_v$ and $Y_{I_w}=\sum_{v\in W}a_{w,v}\de_v$. Let $\de_w=\sum_{v\in W}b_{w,v}Y_{I_v}$ so that $\sum_{v\in W}a_{w,v}b_{v,u}=\Kr_{w,u}$ and $Y_{I_u}^*=\sum_{v\in W}b_{v,u}f_v$.

Combining the formula of Lemma~\ref{lem:AIrev} with the formula $A_\Pi(f_v)=\tfrac{1}{v(x_\Pi )}\unit$ of Lemma~\ref{lem:APonfv}, we obtain
\[
A_\Pi(Y_{I_u}^* A_{I_w^{\rev}}(x_\Pi f_e))=\sum_{v\in W} b_{v,u}v(x_\Pi)a_{w,v}A_\Pi(f_v)=\sum_{v\in W} b_{v,u}a_{w,v}\unit=\Kr_{w,u}\unit.\qedhere
\]
\end{proof}


\section{An involution} \label{sec:invol}

In the present section we define an involution on $\DcF$ and study the relationship between the equivariant characteristic map and the push-pull operators.

\medskip

We define an $R$-linear involution $\tau: \QW \to \QW$ by 
\[
\tau(q \de_w) 
= w^{-1}(q) \tfrac{x_\Pi}{w^{-1}(x_\Pi)} \de_{w^{-1}}
= x_\Pi \de_{w^{-1}} q \tfrac{1}{x_\Pi} 
= \de_{w^{-1}} q \tfrac{w(x_\Pi)}{x_\Pi}
\] 
in particular, $\tau(X_\al)=X_\al$ and $\tau(Y_\al)=Y_\al$.
\begin{lem}\label{lem:inv}
We have $\tau(z_1z_2)=\tau(z_2)\tau(z_1)$ for any $z_1,z_2 \in \QW$, \ie the map $\tau$ just defined is indeed an involution.
\end{lem}
\begin{proof}
For any $q \in Q$, we have $\tau(q)=q$ and $\tau(q \de_w) = \tau(\de_w) q$, so it suffices to check that $\tau(\de_v \de_w)=\tau(\de_w) \tau(\de_v)$, which it is immediate from the definition of the multiplication in $\QW$. 
\end{proof}

Note that $\tfrac{x_\Pi}{w^{-1}(x_\Pi)}$ is in $S$ for any $w\in W$ by Lemma \ref{lem:Sigmaw}.\eqref{item:cw0Sigma}, so the involution $\tau$ restricts to $\SW$. 
\begin{cor}\label{cor:invXY}For any sequence $I$, we have $\tau(X_i)=X_i$ and $\tau(qX_I)=X_{I^{\rev}}q$. In particular, $\tau$ induces an involution on $\DcF$.
\end{cor}
\begin{proof}By Lemma \ref{lem:inv} it suffices to show that $\tau(X_i)=X_i$, which follows from direct computation.
\end{proof}

Recall that the characteristic map $c:Q\to \QWd$ introduced in \ref{defi:charmap} satisfies that $q\mapsto \sum_{w\in W}w(q)f_w$, or in other words, $c(q)(z)=z\cdot q$ for $z \in \QW$. In particular, we have 
\[
c(q)(X_I)=\Dem_I(q)\text{ and }c(q)(\de_w)=w(q),\quad w\in W.
\]

\begin{lem}\label{lem:charinvol}For any $q\in Q$ and $z\in Q_W$, we have 
\[
A_\Pi\left((\tau(z)\act \tf_e)c(q)\right)=(z\cdot q)\unit.
\]
\end{lem}
\begin{proof}Let $z=p\de_w$, $p\in Q$, then $\tau(z)\act \tf_e=\de_{w^{-1}}p\frac{w(x_\Pi)}{x_\Pi}\act (x_\Pi f_e)=pw(x_\Pi)f_w$, so 
\[
\begin{split}
A_\Pi\left((\tau(z)\act \tf_e)c(q)\right)&=A_\Pi\left((pw(x_\Pi)f_w)(\sum_{v\in W}v(q)f_v)\right)\\
&=A_\Pi\left(pw(x_\Pi)w(q)f_w\right)=pw(q)\unit=(z\cdot q)\unit.\qedhere
\end{split}
\]
\end{proof}

We have the following special cases of Lemma \ref{lem:charinvol}:
\begin{cor} 
For any sequence $I$ and $x\in S$, we have
\[
A_\Pi(c(q) A_{I^{\rev}}(\tf_e))=C_I(q)\unit ~\text{ and }~ A_\Pi(c(q) B_{I^{\rev}}(\tf_e))=\Dem_I(q)\unit.
\]
\end{cor}

\begin{proof}Letting $z=Y_I$ (resp. $z=X_I$) in Lemma \ref{lem:charinvol}, and using $\tau(Y_I)=Y_{I^{\rev}}$ and $\tau(X_I)=X_{I^{\rev}}$  from Corollary \ref{cor:invXY} we get the two identities. 
\end{proof}

\begin{cor}
For any $z\in \QW$, we have 
$A_\Pi(\tau(z)\act \tf_e)=(z\cdot 1)\unit$. In particular,  $A_\Pi(q\act \tf_e)=q\unit$ and $A_\Pi(B_{I}(\tf_e))=\Dem_{I^{\rev}}(1)\unit=\Kr_{I,\emptyset} \unit$. 
\end{cor}


\section{The non-degenerate pairing on the $W_{\Xi}$-invariant subring} \label{sec:nondeg}

In this section, we construct a non-degenerate pairing on the subring of invariants $(\DcFd)^{W_\Xi}$. Using this pairing we provide several $S$-module bases of $(\DcFd)^{W_\Xi}$.

\medskip

For any $w\in W,u\in W^\Xi$ we set
\[
d^Y_{w,u}=u(x_{\Pi/\Xi})\sum_{v\in W_\Xi}a^Y_{w,uv}, \quad d^X_{w,u}=u(x_{\Pi/\Xi})\sum_{v\in W_\Xi}a^X_{w,uv}, \quad \rho_\Xi=\hspace{-1ex}\prod_{w\in W^\Xi}w(x_{\Pi/\Xi})
\]
where $a_{w,uv}^X$ and $a_{w,uv}^Y$ are the coefficients introduced in Lemma~\ref{lem:Xdelta} and~\ref{lem:Ydelta}.

\begin{lem} \label{lem:AXionfe}
For any $w\in W$ we have 
\[
A_\Xi(A_{I^{\rev}_w}(\tf_e))=\hspace{-1ex}\sum_{u\in W^\Xi}d^Y_{w,u} f_{u}^\Xi, 
\qquad  
A_\Xi(B_{I^{\rev}_w}(\tf_e))=\hspace{-1ex}\sum_{u\in W^\Xi}d^X_{w,u}f_{u}^\Xi.
\]
\end{lem}

\begin{proof}
We prove the first formula only; the second one is obtained similarly. By Lemma~\ref{lem:AIrev} and~\ref{lem:APonfv},
\[
A_\Xi(A_{I^{\rev}_w}(x_\Pi f_e))=A_\Xi(\sum_{v\in W}v(x_\Pi) a^Y_{w,v}f_{v})=\sum_{v\in W}v(x_{\Pi/\Xi})a^Y_{w,v}f_v^\Xi=
\]
by \eqref{prop:coset}, representing $v=uv'$, and Lemma~\ref{lem:PQroot}, 
\[
=\sum_{u\in W^\Xi,\, v'\in W_\Xi} uv'(x_{\Pi/\Xi})a^Y_{w,uv'}f_{uv'}^\Xi=\sum_{u\in W^\Xi,\, v'\in W_\Xi} u(x_{\Pi/\Xi})a^Y_{w,uv'}f_{u}^\Xi. \qedhere
\]
\end{proof}

\begin{lem} \label{lem:coeffXi}
For any $w\in W, u\in W^\Xi$, we have $d^Y_{w,u}$ and $d^X_{w,u}$ belong to $S$.
\end{lem}
\begin{proof}
It follows from Lemma \ref{lem:AXionfe} and the fact $\DcFd\subseteq \SWd$.
\end{proof}

\begin{theo}\label{theo:invbasis}
For any choice of reduced sequences $\{I_w\}_{w\in W^\Xi}$, the two families $\big\{A_\Xi(A_{I_u^{\rev}}(\tf_e))\big\}_{u\in W^\Xi}$ and $\big\{A_\Xi(B_{I_u^{\rev}}(\tf_e))\big\}_{u\in W^\Xi}$ are $S$-module bases of $(\DcFd)^{W_\Xi}$.
\end{theo}

\begin{proof}
Let us first complete our choice of reduced sequence as a $\Xi$-compatible one, by choosing sequences $I_u$ for each $u\in W_\Xi$. By Corollary \ref{cor:opAinv} our families are in the $S$-module $(\DcFd)^{W_\Xi}$. To show that they are bases, it suffices to show that the respective matrices $M^Y_\Xi$ and $M^X_\Xi$ expressing them on the basis $\{X_{I_u}^*\}_{u \in W^{\Xi}}$ of Lemma \ref{lem:DdWXibasis} have invertible determinants (in $S$).

If $u'\in W^\Xi$ and $v\in W_\Xi$, we have $u'\leq u'v$ where the equality holds if and only if $v=e$. By Lemma~\ref{lem:Ydelta}, we get $a^Y_{u,u'v}=0$ unless $u'\le u$ and $a^Y_{u,uv}=0$ if $v \neq e$. This implies that $d^Y_{u,u'}=0$ unless $u'\le u$, and that
\[
d^Y_{u,u}=u(x_{\Pi/\Xi})\sum_{v\in W_\Xi}a^Y_{u,uv}=u(x_{\Pi/\Xi})a^Y_{u,u}=u(x_{\Pi/\Xi})\tfrac{1}{x_u}.
\]
Hence, the matrix $D_\Xi^Y:=(d^Y_{u,u'})_{u,u'\in W^\Xi}$ is lower triangular with determinant $\rho_\Xi \prod_{u\in W^\Xi}\tfrac{1}{x_u}$. Similarly, the matrix $D_\Xi^X:=(d^X_{u,u'})_{u,u'\in W^\Xi}$ is lower triangular with determinant $\rho_\Xi\prod_{u\in W^\Xi}\tfrac{(-1)^{\ell(u)}}{x_u}$. 

On the other hand, for $u\in W^\Xi$, we have 
\[
X_{I_u}^*=\sum_{w\in W}b^X_{w,u}f_w=\sum_{u'\in W^\Xi}\sum_{v\in W_\Xi}b^X_{u'v,u}f_{u'v}.
\]
By Corollary \ref{cor:inv_coeff}, and because $X_{I_u}^*$ is fixed by $W_{\Xi}$, we have $b^X_{u'v,u}=b^X_{u',u}$.  Therefore,
\[
X_{I_u}^*=\sum_{u'\in W^\Xi}b^X_{u',u}\sum_{v}f_{u'v}=\sum_{u'\in W^\Xi}b^X_{u',u}f_{u'}^\Xi.
\]
By Lemma \ref{lem:Xdelta}, $b^X_{u',u}=0$ unless $u'\ge u$, so the matrix $E_\Xi^X:=\{b^X_{u',u}\}_{u',u\in W^\Xi}$ is lower triangular with determinant 
$\prod_{u\in W^\Xi}(-1)^{\ell(u)}x_u$.

The matrix $M^X_\Xi=(E_\Xi^X)^{-1} D_\Xi^X$ has determinant 
\[
\rho_\Xi\prod_{u \in W^\Xi} \tfrac{1}{(x_u)^2}
\] 
which is invertible in $S$ by Lemma~\ref{lem:roottwo} below. Since the determinant of $M^Y_\Xi=(E_\Xi^X)^{-1} D_\Xi^Y$ differs by sign only, it is
invertible as well.
\end{proof}

Recall the definition of $\RS_\Xi$ from the beginning of section \ref{sec:pushpull}, and let $w_{0,\Xi}$ be the longest element of $W_\Xi$. 
\begin{lem} \label{lem:bww0}
For any $w\in W_\Xi$, we have $x_wx_{ww_{0,\Xi}}=w_{0,\Xi}(x_\Xi)$. In particular, if $\Xi=\Pi$ we have $x_{w}x_{ww_0}=x_{w_0}$.
\end{lem}

\begin{proof}
Recall from Lemma \ref{lem:Ydelta} that $b^Y_{w,w}=x_w=\prod_{w\RS^-\cap \RS^+} x_\al$. By \eqref{eq:root}, it also equals $\prod_{w\RS_\Xi^-\cap \RS_\Xi^+} x_\al$. Since $w_{0,\Xi}\RS_\Xi^-=\RS_\Xi^+$, we have $ww_{0,\Xi}\RS_\Xi^-\cap \RS_\Xi^+=w\RS_\Xi^+\cap \RS_\Xi^+$. Moreover,   
\[
(w\RS_\Xi^-\cap \RS_\Xi^+)\cap (w\RS_\Xi^+\cap \RS_\Xi^+)\subset w\RS_\Xi^-\cap w\RS_\Xi^+=w(\RS_\Xi^-\cap \RS_\Xi^+)=\emptyset
\]
and their union is $\RS_\Xi^+$.
\end{proof}

\begin{lem} \label{lem:roottwo}
For any $\Xi\subset \Pi$ the product $\rho_\Xi\prod_{u\in W^\Xi} \tfrac{1}{x_u^2}$ is an invertible element in $S$. 
\end{lem}

\begin{proof}
We already know that this product is in $S$, since it is the determinant of the matrix $M_\Xi^X$ whose coefficients are in $S$. Consider the $R$-linear involution $u\mapsto \bar u$ on $S=\RcF$ induced by $\la \mapsto -\la$, $\la\in \cl$. Observe that it is $W$-equivariant. 

 For any $\al\in \Xi$, we have 
\[
x_\Xi=s_\al(x_\Xi) x_{-\al} x_\al^{-1} = s_\al(x_\Xi) \bar x_{\al} x_{\al}^{-1}
\]
and, therefore, by induction $x_\Xi=w(x_\Xi)\bar x_vx_v^{-1}$ for any $v\in W_\Xi$. In particular,  $x_\Pi=w(x_\Pi)\bar x_w x_w^{-1}$ for any $w\in W$. Then 
\[
x_\Xi^{|W_\Xi|}=\prod_{v\in W_\Xi}v(x_\Xi)\bar x_v x_v^{-1}\quad\text{and}\quad x_\Pi^{|W|}=\prod_{w\in W}w(x_\Pi)\bar x_w x_w^{-1}.
\]
If $w=uv$ with $\ell(w)=\ell(u)+\ell(v)$, by Lemma \ref{lem:Sigmaw}, part \eqref{item:Sigmauv}, $x_{uv}=x_uu(x_v) $ and $\bar x_{uv}=\bar x_u u(\bar x_v )$. 
Hence
\begin{equation} \label{eq:xpi}
\begin{split}
x_\Pi^{|W|} & =\prod_{w\in W}w(x_\Pi)\bar x_w x_w^{-1}=\prod_{u\in W^\Xi}\prod_{v\in W_\Xi}uv(x_{\Pi/\Xi}x_\Xi)\bar x_{uv} x_{uv}^{-1} \\
& \hspace{-.2ex}\equalbyref{cor:xXifixed} \prod_{u\in W^\Xi} u\big(x_{\Pi/\Xi}^{|W_\Xi|}\big) \prod_{v\in W_\Xi} uv(x_\Xi)\bar x_u u(\bar x_v)x_u^{-1}u(x_v^{-1}) \\
& = \rho_\Xi^{|W_\Xi|} \prod_{u\in W^\Xi} \big(\bar x_u x_u^{-1}\big)^{|W_\Xi|} u\big( \prod_{v\in W_\Xi} v(x_\Xi) \bar x_v x_v^{-1}\big) \\
& = \rho_\Xi^{|W_\Xi|} \prod_{u\in W^\Xi} \big(\bar x_u x_u^{-1}\big)^{|W_\Xi|} u( x_\Xi)^{|W_\Xi|} 
\end{split}
\end{equation}

On the other hand, by Lemma \ref{lem:bww0}, 
\[
\bar x_\Xi^{|W_\Xi|}=w_{0,\Xi}(x_\Xi)^{|W_\Xi|}=\prod_{v \in W_\Xi}x_v x_{vw_{0,\Xi}}=\prod_{v\in W_\Xi}x_v^2
\]
and, in particular, $\bar x_\Pi^{|W|}=\prod_{w\in W} x^2_w$. So, we obtain
\[
\begin{split}
\bar x_\Pi^{|W|} & =\prod_{w\in W}x_{w}^2=\prod_{u\in W^\Xi}\prod_{v\in W_\Xi}x_{uv}^2=\prod_{u\in W^\Xi}\prod_{v\in W_\Xi}x^2_{u}u(x^2_{v}) \\
& =\Big(\prod_{u\in W^\Xi}x_{u}^{2|W_\Xi|}\Big) \Big(\prod_{u\in W^\Xi}u(\prod_{v\in W_\Xi}x^2_{v})\Big)=\Big(\prod_{u\in W^\Xi}x_{u}^2\Big)^{|W_\Xi|} \Big(\prod_{u\in W^\Xi}u(\bar x_{\Xi})\Big)^{|W_\Xi|}. 
\end{split}
\]
Combining this with equation \eqref{eq:xpi}, we obtain 
\[
\Big(\rho_\Xi^{-1} \prod_{u\in W^\Xi}x_u^2\Big)^{|W_\Xi|} =\bar x_\Pi^{|W|}x_\Pi^{-|W|}\Big(\prod_{u\in W^\Xi} u(\bar x_{\Xi}^{-1}x_\Xi)\bar x_u x_u^{-1}\Big)^{|W_\Xi|}
\]
which is an element of $S$, since it is a product of elements of the form $x_\al x_{-\al}^{-1} \in S$. Therefore $\rho_\Xi\prod_{u \in W^\Xi} \tfrac{1}{x_u^2}$ is invertible, since so is its $|W_\Xi|$-th power.
\end{proof}

\begin{cor}\label{cor:Aopsurj} 
Given $\Xi'\subseteq \Xi\subseteq \Pi$ we have $A_{\Xi}(\DcFd)=(\DcFd)^{W_\Xi}$. 
For any set of coset representatives $W_{\Xi/\Xi'}$ the operator $A_{\Xi/\Xi'}$ induces  a surjection $(\DcFd)^{W_\Xi'}\to (\DcFd)^{W_\Xi}$ (independent of the choices of $W_{\Xi/\Xi'}$ by Lemma \ref{lem:Ainvtoinv}).
\end{cor}

\begin{proof}
By Corollary~\ref{cor:opAinv} and Theorem~\ref{theo:invbasis}, we obtain the first part. To prove the second part, let $\sigma\in (\DcFd)^{W_{\Xi'}}$. By the first part, there exists $\sigma'\in \DcFd$ such that $\sigma=A_{\Xi'}(\sigma')$, so by Lemma \ref{lem:compositA} we have 
\[
A_{\Xi/\Xi'}(\sigma)=A_{\Xi/\Xi'}(A_{\Xi'}(\sigma'))=A_{\Xi}(\sigma')\in (\DcFd)^{W_\Xi}.
\] 
Hence, $A_{\Xi/\Xi'}$ restricts to $A_{\Xi/\Xi'}\colon (\DcFd)^{W_{\Xi'}}\to (\DcFd)^{W_\Xi}$. Since $A_\Xi(\DcFd)=(\DcFd)^{W_\Xi}$, we also have $A_{\Xi/\Xi'}((\DcFd)^{W_{\Xi'}})=(\DcFd)^{W_\Xi}$. 
\end{proof}

\begin{theo}\label{theo:bilformGP} 
Assume that the choice of reduced sequences $\{I_w\}_{w\in W}$ is $\Xi$-compa\-tible. If $u\in W^\Xi$, then 
\[
A_{\Pi/\Xi}(X_{I_u}^\star A_\Xi(B_{I_w}^{\rev}(x_\Pi f_e)))=\Kr_{w,u}\unit.
\]
Consequently, the pairing 
\[
(\DcFd)^{W_\Xi}\times (\DcFd)^{W_\Xi}\to (\DcFd)^W\cong S,\quad (\sigma,\sigma')\mapsto A_{\Pi/\Xi}(\sigma\sigma')
\]
is non-degenerate; 
$\big\{A_{\Xi}(B_{I_u^{\rev}}(x_\Pi f_e))\big\}_{u\in W^\Xi}$ and $\{X_{I_u}^\star\}_{u\in W^\Xi}$ being dual $S$-bases of $(\DcFd)^{W_\Xi}$.
\end{theo}

\begin{proof}
By Corollary \ref{cor:Aopsurj}, the pairing is well-defined (\ie it does map into $S$). By Lemma~\ref{lem:projformulaA}, Lemma~\ref{lem:compositA} and Theorem~\ref{theo:bilform}, we obtain $A_{\Pi/\Xi}(X_{I_u}^*A_{\Xi}(B_{I_w^{\rev}}(x_\Pi f_e)))  =$
\[
 =A_{\Pi/\Xi}(A_{\Xi}(X_{I_u}^* B_{I_w^{\rev}}(x_{\Pi}f_e))) 
= A_\Pi(X_{I_u}^*B_{I_w^{\rev}}(x_\Pi f_e)) =\Kr_{w,u}\unit.\qedhere
\]
\end{proof}


\section{Push-forwards and pairings on $\DcFPd{\Xi}$} \label{sec:pushDXi}

We construct now an algebraic version of the push-forward map.

\medskip

For any $\Xi \subseteq \Pi$, the $W_\Xi$ invariant subring $S^{W_\Xi}$ (resp. $Q^{W_\Xi}$) acts by multiplication on the right on $\SWP{\Xi}$ (resp. $\QWP{\Xi}$) by the formula $(\sum_{\bar{w}} q_{\bar{w}} \de_{\bar{w}}) \cdot q' = \sum_{\bar{w}} q_{\bar{w}} w(q') \de_{\bar{w}}$ (note that $w(q')$ does not depend on the choice of a representative $w$ of $\bar{w}$). When $q \in S^{W_\Xi}$ (resp. $Q^{W_\Xi}$) and $f \in \SWPd{\Xi}$ (resp. $f \in \QWPd{\Xi}$), we write $q \act f$ for the map dual to the multiplication on the right by $q$.

Recall that $\ddd{\Xi}:\QWPd{\Xi'} \to \QWPd{\Xi}$ was defined at the beginning of section \ref{sec:algresP}, and that it sends $f_{\tilde{w}}$ to $f_{\bar{w}}$. By Corollary \ref{cor:xXifixed} we know that $\frac{1}{x_{\Xi/\Xi'}}\in (Q)^{W_{\Xi'}}$.

\medskip

We define $\cA_{\Xi/\Xi'}\colon \QWPd{\Xi'} \to \QWPd{\Xi}$ by $\cA_{\Xi/\Xi'}(f):=\ddd{\Xi/\Xi'}\big((1/x_{\Xi/\Xi'})\act f\big)$.
The left commutative diagram
\[
\xymatrix{
\QWP{\Xi'} & \QW \ar[l]_-{\p{\Xi'}} \\ 
\QWP{\Xi} \ar[u]^{\cdot \frac{1}{x_{\Xi/\Xi'}} \circ \dd{\Xi/\Xi'}} & \QW \ar[l]_-{\p{\Xi}} \ar[u]_{\cdot Y_{\Xi/\Xi'}}
}
\hspace{5ex}
\xymatrix{
\QWPd{\Xi'} \ar[d]_{\cA_{\Xi/\Xi'}} \ar[r]^-{\pd{\Xi'}} & \QWd \ar[d]^{A_{\Xi/\Xi'}} \\
\QWPd{\Xi} \ar[r]^-{\pd{\Xi}} & \QWd
}
\]
in which $\cdot 1/x_{\Xi/\Xi'}$ and $\cdot Y_{\Xi/\Xi'}$ mean multiplication on the right, dualizes as the right one. Since $\pd{\Xi}$ restricts to an isomorphism $\DcFPd{\Xi} \isoto (\DcFd)^{W_\Xi}$ by Lemma \ref{lem:DFinv} and since $A_{\Xi/\Xi'}$ restricts to a  map $(\DcFd)^{W_{\Xi'}}\to (\DcFd)^{W_\Xi}$ by Corollary \ref{cor:Aopsurj}, we obtain:
\begin{lem}\label{lem:AXicommute}
The map $\cA_{\Xi/\Xi'}$ restricts to $\DcFPd{\Xi} \to \DcFPd{\Xi'}$ and the diagram
\[
\xymatrix{
\DcFPd{\Xi'} \ar[d]_{\cA_{\Xi/\Xi'}} \ar[r]^-{\pd{\Xi'}}_-{\simeq} & (\DcFd)^{W_{\Xi'}} \ar[d]^{A_{\Xi/\Xi'}} \\
\DcFPd{\Xi} \ar[r]^-{\pd{\Xi}}_-{\simeq} & (\DcFd)^{W_{\Xi}}
}
\]
commutes.
\end{lem}

\begin{rem}
The map $\cA_{\Xi/\Xi'}$ corresponds to a push-forward in the geometric context, see \cite[Diagram~(8.3)]{CZZ2}
\end{rem}

\begin{lem}\label{lem:multDSP} 
Within $\QWP{\Xi}$, we have $\DcFP{\Xi}x_{\Pi/\Xi}\subseteq \SWP{\Xi}$. So the right multiplication by $x_{\Pi/\Xi}$ induces a map $\DcFP{\Xi}\to \SWP{\Xi}$. Consequently, it defines a map $\SWPd{\Xi}\to \DcFPd{\Xi}$, $f\mapsto x_{\Pi/\Xi}\act f$.
\end{lem}

\begin{proof}
By Lemma \ref{lem:DcFPbasis} we know that $\{X_{I_w}^\Xi\}_{w\in W^\Xi}$ is a basis of $\DcFP{\Xi}$, so it suffices to show that $X_{I_w}^\Xi x_{\Pi/\Xi}\in \SWP{\Xi}.$ We have 
\[
X_{I_w}^\Xi x_{\Pi/\Xi}=\hspace{-1ex}\sum_{u\in W^\Xi}(\sum_{v\in W_\Xi}a^X_{w,uv})\de_{\bar u}x_{\Pi/\Xi}=\hspace{-1ex}\sum_{u\in W^\Xi}(\sum_{v\in W_\Xi}u(x_{\Pi/\Xi})a^X_{w,uv})\de_{\bar u}=\hspace{-1ex}\sum_{u\in W^\Xi}d^Y_{w,u}\de_{\bar u},
\]
which belongs to $\SWP{\Xi}$ by Lemma \ref{lem:coeffXi}.
\end{proof}
The geometric translation of the map $\SWPd{\Xi}\to \DcFPd{\Xi}$ is the push-forward map from the $T$-fixed points of $G/P_\Xi$ to $G/P_\Xi$, see \cite[Diagram~(8.1)]{CZZ2}.

\begin{example}
Note that in general $x_{\Pi/\Xi}\DcFP{\Xi}\not\subseteq\SWP{\Xi}$. For example, let the root datum be of type $A_2^{ad}$ and $\Xi=\{\alpha_2\}$, then $x_{\Pi/\Xi}=x_{-\al_1}x_{-\al_1-\al_2}$. Let $w=s_2s_1\in W^\Xi$, then 
\[
X_{21}=\tfrac{1}{x_{\al_1}x_{\al_2}}\de_e-\tfrac{1}{x_{\al_2}x_{\al_1+\al_2}}\de_{s_2}-\tfrac{1}{x_{\al_1}x_{\al_2}}\de_{s_1}+\tfrac{1}{x_{\al_2}x_{\al_1+\al_2}}\de_{s_2s_1}.
\]
Then  $X^\Xi_{21}x_{\Pi/\Xi}\in \SWP{\Xi}$ but $x_{\Pi/\Xi}X^\Xi_{21}\not\in \SWP{\Xi}$.
\end{example}
One easily checks that the diagram on the left below is commutative, and it restricts as the one on the right by Lemma \ref{lem:multDSP}.
\[
\xymatrix{
\QWP{\Xi'} \ar[r]^-{\cdot x_{\Pi/\Xi'}} & \QWP{\Xi'} \\
\QWP{\Xi} \ar[u]^{\cdot \frac{1}{x_{\Xi/\Xi'}} \circ \dd{\Xi/\Xi'}} \ar[r]^-{\cdot x_{\Pi/\Xi}} & \QWP{\Xi} \ar[u]_{\dd{\Xi/\Xi'}}
}
\qquad
\xymatrix{
\DcFP{\Xi'} \ar[r]^-{\cdot x_{\Pi/\Xi'}} & \SWP{\Xi'} \\
\DcFP{\Xi} \ar[u]^{\cdot \frac{1}{x_{\Xi/\Xi'}} \circ \dd{\Xi/\Xi'}} \ar[r]^-{\cdot x_{\Pi/\Xi}} & \SWP{\Xi} \ar[u]_{\dd{\Xi/\Xi'}}
}
\]
By $S$-dualization, one obtains the commutative diagram
\[
\xymatrix{
\DcFPd{\Xi'} \ar[d]_{\cA_{\Xi/\Xi'}} & \SWPd{\Xi'} \ar[l]_-{x_{\Pi/\Xi'} \act}  \ar[d]^{\ddd{\Xi/\Xi'}} \\
\DcFPd{\Xi} & \SWPd{\Xi} \ar[l]^-{x_{\Pi/\Xi} \act} 
}
\]
whose geometric interpretation in terms of push-forwards is given in \cite[Diagram~(8.3)]{CZZ2}

Finally, Theorems \ref{theo:bilform} and \ref{theo:bilformGP} immediately translate as:

\begin{theo}
The pairing $\DcFd \times \DcFd \to \DcFPd{\Pi}\simeq S$ defined by sending $(\sigma,\sigma')$ to $\cA_{\Pi}(\sigma \sigma')$ is non degenerate; $\big\{A_{I_w^\rev}(x_\Pi f_e)\big\}_{w \in W}$ and $\{Y_{I_v}^\star\}_{v \in W}$ are dual bases and so are $\big\{B_{I_w^\rev}(x_\Pi f_e)\big\}_{w \in W}$ and $\{X_{I_v}^\star\}_{v \in W}$.
\end{theo}

\begin{theo} \label{theo:pairingDXi}
The pairing $\DcFPd{\Xi} \times \DcFPd{\Xi} \to \DcFPd{\Pi}\simeq S$ defined by sending $(\sigma,\sigma')$ to $\cA_{\Pi/\Xi}(\sigma \sigma')$ is non degenerate, and $\big\{\cA_{\Xi}(B_{I_w^\rev}(x_\Pi f_e))\big\}_{w \in W^{\Xi}}$ and $\big\{(X_{I_v}^\Xi)^\star\big\}_{v \in W^\Xi}$ are dual bases.
\end{theo}

\begin{proof}For any choice of $\{I_w\}_{w\in W^\Xi}$, we complete it into a $\Xi$-compatible family $\{I_w\}_{w\in W}$, then by Lemma \ref{lem:DdWXibasis} $\{X^*_{I_w}\}_{w\in W^\Xi}$ is a basis of $(\DcFd)^{W_\Xi}$. By Lemma \ref{lem:DFinv} we know that $\pd{\Xi}((X^\Xi_{I_w})^\star)=X^*_{I_w}$ if $w\in W^\Xi$, so the conclusion follows from Lemma \ref{lem:AXicommute} and Theorem \ref {theo:bilformGP}.
\end{proof}

In some sense, Theorem \ref{theo:pairingDXi} is not completely satisfactory in terms of geometry: in the parabolic case, although we do know that the Schubert classes $\{\cA_{\Xi}A_{I_w^\rev}(x_\Pi f_e)\}_{w \in W^\Xi}$ form a basis, we did not find a good description of the dual basis with respect to the bilinear form.

\end{document}